\documentclass[a4paper,reqno]{amsart}
\usepackage{amsfonts}
\usepackage{amsaddr}
\usepackage{color}
\usepackage{hyperref}
\definecolor{ddorange}{rgb}{1,0.5,0}
\definecolor{ddcyan}{rgb}{0,0.2,1.0}
\newcommand{\GGG}{\color{black}}

\newcommand{\BBB}{\color{black}}
\newcommand{\EEE}{\color{black}}
\usepackage[textwidth=408pt,textheight=658pt,heightrounded, headsep=12pt,vmarginratio=1:1]{geometry}
\usepackage[T1]{fontenc} 
\usepackage[utf8]{inputenc}
\usepackage{amsmath,amssymb,amsthm,mathrsfs,esint}
\usepackage{mathtools}

\newcommand{\R}{{\mathbb R}}

\newcommand{\N}{{\mathbb N}}

\newcommand{\I}{{\mathrm{I}}}

\newcommand{\Rn}{{\R}^n}
\newcommand{\xy}{^\xi_y}
\newcommand{\xoy}{^{\xi_0}_y}

\newcommand{\Oxy}{{\Omega^\xi_y}}

\newcommand{\dx}{\, \mathrm{d} x}
\renewcommand{\dh}{\, \mathrm{d} \mathcal{H}^{n-1}}
\newcommand{\hn}{\mathcal{H}^{n-1}}
\newcommand{\Ln}{{\mathcal{L}}^n}
\newcommand{\ho}{\mathcal{H}^0}
\newcommand{\Sn}{{\mathbb{S}^{n-1}}}
\newcommand{\ol}{\overline}

\newcommand{\wt}{\widetilde}

\newcommand{\sm}{\setminus}
\newcommand{\Mnn}{{\mathbb{M}^{n\times n}_{sym}}}
\newcommand{\Mskew}{{\mathbb{M}^{n\times n}_{\rm skew}}}
\newcommand{\dod}{{\partial_D \Omega}}

\newcommand{\dom}{{\partial \Omega}}
\newcommand{\weak}{\rightharpoonup}
\newcommand{\wstar}{\stackrel{*}\rightharpoonup}

\newcommand{\mres}{\mathbin{\vrule height 1.6ex depth 0pt width
0.13ex\vrule height 0.13ex depth 0pt width 1.3ex}}

\DeclareMathOperator*{\aplim}{ap\,lim}

\theoremstyle{plain}
\begingroup
\theoremstyle{plain}
\newtheorem{theorem}{Theorem}[section]

\newtheorem{proposition}[theorem]{Proposition}
\newtheorem{lemma}[theorem]{Lemma}
\theoremstyle{definition}
\newtheorem{definition}[theorem]{Definition}
\theoremstyle{remark}
\newtheorem{remark}[theorem]{Remark}

\endgroup

\numberwithin{equation}{section}
\newcommand{\Addresses}{{
  \bigskip
  \footnotesize
(A.~Chambolle) \textsc{CEREMADE, CNRS and Université Paris-Dauphine PSL, Paris, France}
\par\nopagebreak
  \textit{E-mail address}, A.~Chambolle: \texttt{antonin.chambolle@ceremade.dauphine.fr}

\medskip
(V.~Crismale)  \textsc{Dipartimento di Matematica, Università Roma-I “La Sapienza”, 00185 Roma, Italy.}\par\nopagebreak
  \textit{E-mail address}, V.~Crismale: \texttt{vito.crismale@uniroma1.it}

}}

\title
[Equilibrium in nonhomogeneous linear elasticity with discontinuities]
{Equilibrium configurations for nonhomogeneous linearly elastic materials with surface discontinuities
}

\author{Antonin Chambolle \and Vito Crismale}
\email[Antonin Chambolle]{antonin.chambolle@ceremade.dauphine.fr}
\email[Vito Crismale]{vito.crismale@uniroma1.it}

\begin{document}
\begin{abstract}
We prove a compactness and semicontinuity result that applies to minimisation problems in nonhomogeneous linear elasticity under Dirichlet boundary conditions. This generalises a previous compactness theorem that we proved and employed to show existence of minimisers for the Dirichlet problem for the (homogeneous) Griffith energy.
\end{abstract}

\maketitle

\setcounter{tocdepth}{1}  
\tableofcontents

\section{Introduction}
In this paper we study the minimisation of free discontinuity functionals describing energies for
linearly elastic solids with discontinuities, under Dirichlet boundary conditions.
For a solid in a given (bounded) reference configuration $\Omega \subset \Rn$, whose \emph{displacement field} with respect to the equilibrium is $u\colon \Omega \to \Rn$, the minimisation of integral functionals of the form
\begin{equation}\label{1705201919}
E(u):= \int_\Omega f(x,e(u))\dx + \int_{J_{u}} g( x, \EEE [u], \nu_u) \dh
\end{equation}
accounts for the interaction of the internal elastic energy
and the energy dissipated in the surface discontinuities.

The elastic properties of the solid are determined by 
the \emph{elastic strain} $e(u)=\tfrac12 (\nabla u + (\nabla u)^{\mathrm{T}})$, the symmetrized gradient of $u$, through a function $f$ with superlinear growth in $e(u)$ (often a quadratic form) and in general depending on the material point $x\in \Omega$. The surface term is related  to dissipative phenomena such as cracks, surface tension between different elastic phases, or internal cavities, and is concentrated on the \emph{jump set} $J_u$, representing the surface discontinuities of $u$. 
The jump set is such that
when blowing up around any $x \in J_u$, it is
approximated by a hyperplane with  normal $\nu_u(x)\in \Sn$ and the displacement field is close to two suitable distinct values $u^+(x)$, $u^-(x) \in \Rn$ on the two sides of the body with respect to this hyperplane. The jump opening, denoted by $[u]$, is then $[u](x)=u^+(x)-u^-(x)$. In order to ensure that the volume and the surface term do not interact, it is usually assumed that $g$ be greater than a positive constant, or 
some 
growth condition for small values of $[u]$
(besides the superlinear growth of $f$).
Therefore, the functionals we consider are bounded from below
 through
 the Griffith-like energy (\cite{Griffith, FraMar98})
\begin{equation}\label{1805200928}
G(u):=\int_\Omega |e(u)|^p \dx + \hn(J_u)\,, \qquad\text{with }p>1\,.
\end{equation}

The first main issue in the minimisation of energies of the type \eqref{1705201919} when also the control from above is only through \eqref{1805200928} (in particular if $g$ is independent of $[u]$) is how to obtain suitable compactness.
This is related to the lack of good a priori integrability properties for displacements with finite 
energy $G$.
In fact, a pathological situation may occur
 in the presence of connected components, well included in
 $\Omega$, whose boundary is contained (or almost completely contained) in $J_u$: this allows to modify the displacement in these internal components by adding any constant, so that arbitrarily large values of $u$ 
may be reached without  (or slightly) modifying
 the energy. 

Compactness results for sequences with equibounded 
energy \eqref{1805200928} have been obtained with increasing generality. In \cite{BelCosDM98} compactness \BBB w.r.t.\ (with respect to) \EEE strong $L^1$ convergence
is obtained assuming a uniform $L^\infty$ bound on the displacement field: 
this guarantees that the distributional symmetrized gradient $\mathrm{E}u$ is a bounded Radon measure and then $u$ belongs to the space $SBD(\Omega)$ of \emph{special functions of bounded deformation}  \cite{AmbCosDM97}, and in particular $u\in L^1(\Omega;\Rn)$.
In \cite{DM13}, {\sc Dal Maso} introduced the space of \emph{generalised special functions of bounded deformation} $GSBD(\Omega)$ (with the smaller $GSBD^p(\Omega)$, the right energy space for \eqref{1805200928}, see Section~\ref{Sec1}) and proved a compactness result under a uniform mild integrability control on sequences with bounded energy, ensuring convergence in measure.

The first compactness result for \eqref{1805200928} without further constraints is obtained by {\sc Friedrich} \cite{FriPWKorn} \emph{in dimension two}, basing on a \emph{piecewise Korn inequality}. This inequality permits to
ensure the compactness for sequences with bounded energy, up to subtracting suitable \emph{piecewise rigid motions}, namely functions coinciding with an \emph{infinitesimal rigid motion} (that is an affine function with skew-symmetric gradient) on each element of  a suitable Caccioppoli partition $\mathcal{P}=(P_j)_j$ of the domain (that is $\partial^* \mathcal{P}=\bigcup_j \partial^* P_j$ has
 finite surface measure; see \cite{ChaGiaPon07} characterising piecewise rigid motions).

In \cite{CC18} we proved in any dimension that each sequence $(u_h)_h$ with equibounded energy \eqref{1805200928} converges in measure (up to subsequences) to a $GSBD^p$ function $\ol u$, outside an exceptional set with finite perimeter $A$ where $|u_h|\to +\infty$.
 Outside the exceptional set, 
 weak $L^p$ convergence for the symmetrized gradients $(e(u_h))_h$ holds and $\hn(J_{\ol u} \cup \BBB (\partial^* A  \cap \Omega)) \EEE \leq \liminf_{h} \hn(J_{u_h})$.  The main ingredient for basic compactness \BBB w.r.t.\ \EEE the convergence in measure is the  Korn-Poincaré inequality for function with small jump set proven in \cite{CCF16}, while the semicontinuity properties are obtained through a slicing argument. In particular, this directly solves the Dirichlet minimisation problem 
 for the energy \eqref{1805200928},
 with volume term possibly convex with $p$-growth in $e(u)$, but still 
 attaining its minimum value for $e(u)=0$: starting from a minimising sequence $(u_h)_h$, a minimiser is given by any function equal to $\ol u$ in $\Omega\sm A$ and to an infinitesimal rigid motion in $A$. One may argue analogously if the minimum value of $f$ is independent of $x$.
 
However, for general nonhomogeneous materials (for instance composite materials) such that the minimum value of $f(x,\cdot)$ depends on $x$, this strategy does not work and a better characterisation of the limit behaviour also in the exceptional set is required.
A similar issue arises when employing the compactness result by {\sc Ambrosio} \cite{Amb89UMI, Amb90GSBV, AmbNew95} in the space of \emph{generalised functions of bounded variation} $GSBV$, and in its subspace $GSBV^p$ to the minimisation of energies 
\begin{equation}\label{1805201919}
\int_\Omega f(x,\nabla(u))\dx + \int_{J_{u}} g( x, \EEE [u], \nu_u) \dh
\end{equation}
depending on the full gradient $\nabla u$ in place of $e(u)$.
For this reason a compactness result in $GSBV^p$ of different type has been derived in \cite{Fri19}: for any sequence with bounded energy $(u_h)_h$ \eqref{1805201919} it is possible to find modifications $y_h$ 
such that the energy increases at most by $\frac{1}{h}$, $\Ln(\{\nabla u_h \neq \nabla y_h\})\leq \frac{1}{h}$, and $(y_h)_h$ converges in measure to some $u \in GSBV^p$. The functions $y_h$ are indeed obtained from $u_h$ by subtracting a piecewise constant function \emph{up to a set of small measure}, in the same spirit of the aforementioned
 \cite{FriPWKorn}  with piecewise rigid motions replaced by piecewise constant functions.

The present work is based on a different approach: we prove that, given $(u_h)_h$ with $\sup_h(G(u_h))<+\infty$, for suitable piecewise rigid motions $a_h$ the sequence $(u_h-a_h)_h$ converges in measure to some $u \in GSBD^p$, such that $G(u)\leq \liminf_h G(u_h)$.
Differently 
from \cite{Fri19}, we have $e(u_h)=e(u_h-a_h)$ since we subtract piecewise rigid motions (without exceptional sets of small measure); this precludes in general $u_h-a_h$ to be a minimising sequence, but nevertheless the lower semicontinuity for the surface part is obtained directly in terms of $J_{u_h}$.
Our compactness result is the following \BBB (we use notation \eqref{1305201259} for Caccioppoli partitions). \EEE

\begin{theorem}\label{thm:main}
Let $p \in (1,+\infty)$ and $\Omega\subset \Rn$ be open, bounded, and Lipschitz. For any sequence $(u_h)_h$  with  $\sup_{h} G(u_h) <+\infty$  \EEE
there exist a subsequence, not relabelled, 
 a Caccioppoli partition $\mathcal{P}=(P_j)_j$ of $\Omega$, a sequence of piecewise  rigid motions \EEE
  $(a_h)_h$ with
\begin{equation}\label{2302202219}
a_h=\sum_{j\in \N} a_h^j \chi_{P_j}\,,
\end{equation}
 and $u\in GSBD^p(\Omega)$ such that 
\begin{subequations}\label{eqs:0203200917}
\begin{equation}\label{2202201909}
|a_h^j(x)-a_h^i(x)| \to +\infty \quad \text{for }\Ln\text{-a.e.\ }x\in \Omega, \text{ for all }i\neq j\,,
\end{equation}
and
\begin{align}
u_h-a_h \to u \quad &\Ln\text{-a.e.\ in }\Omega\,, \label{2202201910}\\
e(u_h)  \rightharpoonup e(u) \quad &\text{in } L^p(\Omega; \Mnn)\,,\label{eq:convGradSym}\\
\hn(J_u \cup \BBB (\partial^* \mathcal{P} \cap \Omega) \EEE )\leq \liminf_{h\to \infty} \,&\hn(J_{u_h})\,.\label{eq:sciSalto}
\end{align}
\end{subequations}
\end{theorem}
The first step of the proof consists in finding a partition $\mathcal{P}$, piecewise rigid motions $a_h$, and $u$ measurable such that \eqref{2302202219}, \eqref{2202201909}, and \eqref{2202201910} hold. In doing this, a fundamental tool is a Korn inequality for functions with small jump set, proven in two dimensions in \cite[Theorem~1.2]{CFI16ARMA} and recently extended to any dimension in \cite{CagChaSca20}.
This permits, for every $\eta>0$, to recover \eqref{2302202219} and \eqref{2202201910}  in a set $\Omega^\eta\subset \Omega$, such that $\Ln(\Omega\sm \Omega^\eta)<\eta$.
Then, the so obtained sequences of infinitesimal rigid motions are regrouped 
in equivalence classes for fixed $\eta$, saying that any $(a_h^i)_h$, $(a_h^j)_h$ (depending on $\eta$) are not equivalent if and only if \eqref{2202201909} holds for $i$, $j$. Finally, we pass to $\eta\to 0$ observing that this procedure is stable when $\eta$ decreases: the objects found in correspondence to $\eta$ coincide with those found for $\eta/2$ on $\Omega^\eta \cap \Omega^{\eta/2}$.

In the second step we prove \eqref{eq:sciSalto} through a slicing  procedure.  The guiding idea is that, if \eqref{2302202219}, \eqref{2202201909}, \eqref{2202201910} hold for $n=1$ (with $\Omega$ a real interval, 
$a_h$ piecewise constant, and $(|\nabla u_h|)_h$ equibounded in $L^p$), then not only any jump point of $u$ is a cluster point for $(J_{u_h})_h$ but this holds also for any point $y\in \BBB \partial^* \mathcal{P} \cap \Omega$: \EEE in fact, by \eqref{2202201909} and \eqref{2202201910}, the functions $u_h$ assume  arbitrarily far values, as $h\to \infty$, in couple of points close to $y$ but on different sides of $\Omega\sm\{y\}$, so $u_h$ have to jump near $y$ for $h$ large.
We conclude by noticing that in view of \eqref{eq:sciSalto} the $a_h$ are indeed piecewise rigid \BBB motions, \EEE 
so 
$\sup_h G(u_h-a_h)<+\infty$ and \eqref{eq:convGradSym} follows from 
 former compactness results. 

Besides compactness, we examine the  semicontinuity properties of \BBB $E$, defined in \eqref{1705201919}. \EEE
The lower semicontinuity of the surface term has been recently established for a large class of densities in  \cite{FriPerSol20}, for sequences equibounded \BBB w.r.t.\ \EEE $G$ \BBB (defined in \eqref{1805200928}) \EEE and converging in measure (and also for functionals defined on piecewise rigid motions), providing a counterpart for the analysis of energies \eqref{1805201919} in \cite{AmbBra90-1, AmbBra90-2}.
We then assume that the surface part is lower semicontinuous \BBB w.r.t.\ \EEE the convergence in measure and move in two directions: we address
the semicontinuity properties 
both of the volume term, and of the surface term \BBB w.r.t.\ \EEE the notion of convergence from Theorem~\ref{thm:main}.
We prove the following result \BBB (see \eqref{eq: weak gsbd convergence} for the definition of weak convergence in $GSBD$ and recall \eqref{1305201259}).  \EEE

\begin{theorem}\label{thm:minimization}
Let $p \in (1,+\infty)$ and $\Omega\subset \Rn$ be open, bounded, and Lipschitz. 
Assume that 
\begin{itemize}
\item[($f_1$)] $f\colon \Omega\times\Mnn \to [0,\infty)$ be a Carathéodory function;
\item[($f_2$)] $f(x, \cdot)$ be symmetric quasi-convex for a.e.\ $x \in \Omega$;
\item[($f_3$)] for suitable $C>0$ and $\phi \in L^1(\Omega)$, it holds 
\begin{equation*}
\frac{1}{C}|\xi|^p \leq f(x,\xi) \leq \phi(x) + C(1+|\xi|^p) \quad \text{for a.e.\ }x \in \Omega \text{ and every }\xi \in \Mnn\,;
\end{equation*}
\end{itemize}
moreover assume that
\begin{itemize}
\item[($g_1$)] $g \colon  \Omega \times \EEE \Rn  \times \Sn \to [c, +\infty)$ be measurable, with $c>0$; 
\item[($g_2$)] $g(x,  y \EEE,\nu)=g(x, - y \EEE,-\nu)$ for any  $x$, \EEE $ y \EEE$, $\nu$;
 \item[($g_3$)] $g(\cdot, y, \nu)$ be continuous,
uniformly \BBB w.r.t.\ \EEE $y\in \Rn$ and $\nu \in \Sn$; \EEE
\item[($g_4$)] for  each $x \in \Omega$ $g_x= g(x, \cdot, \cdot)$ be such that for \EEE any cube $Q$ and any $v_h \to v$ weakly in $GSBD^p(Q)$
\begin{equation*}\label{0603200936}
\int_{J_v} g_{ x \EEE}([v], \nu_v) \dh \leq \liminf_{h\to \infty} \int_{J_{v_h}} g_{ x \EEE}([v_h], \nu_{v_h}) \dh \,;
\end{equation*}
\item[($g_5$)] there is  $g_\infty\colon \Omega \times \Rn \to [0,+\infty]$ such that either $g_\infty \equiv+\infty$ or $g_\infty(x,\cdot)$ is a norm for every $x\in \Omega$, and
 \begin{equation*}\label{2703201043}
\lim_{| y \EEE|\to +\infty} g( x,  y \EEE,\nu)=g_{\infty}(x, \nu) \quad\text{uniformly w.r.t.\ } x \in \Omega\text{ and } \EEE\nu \in \Sn\,.
\end{equation*} 
\end{itemize}\EEE
Then, 
for any sequence $(u_h)_h$ such that 
$\sup_{h}E(u_h) < +\infty$ 
there exist a subsequence (not relabelled), a Caccioppoli partition $\mathcal{P}$ of $\Omega$, a sequence of piecewise infinitesimal rigid motions $(a_h)_h$, and $u \in GSBD^p(\Omega)$ such that \eqref{2302202219}, \eqref{eqs:0203200917} hold and 
\begin{equation*}\label{2103201208}
\int_\Omega f(x,e(u))\dx + \int_{J_u \cap \mathcal{P}^{(1)}} g( x \EEE, [u], \nu_u) \dh + \int_{\BBB\partial^* \mathcal{P}\cap \Omega\EEE} g_\infty( x, \EEE\nu_{\mathcal{P}}) \dh \leq \liminf_{h\to \infty} E(u_h)
\end{equation*}
if $g_\infty$ is finite,
while $\hn(\BBB\partial^* \mathcal{P}\cap \Omega\EEE)=0$ 
and $E(u)  \leq \liminf_{h} E(u_h)$ if $g_\infty\equiv +\infty$.
\end{theorem}


The proof relies on a blow-up argument (\cite{FonMue92}).
For the bulk part we use again the result in  \cite{CagChaSca20}. Blowing up around a point $x_0 \notin J_u$, since the density of jump vanishes in $\hn$-measure, the Korn-type inequality of \cite{CagChaSca20} give that the rescaled function coincides with a  $W^{1,p}$ field up to a small set; we then combine this with an approximation through equi-Lipschitz functions, in the footsteps of \cite{AceFus84, Amb94, Ebo05}, in order to apply Morrey's Theorem \cite{Mor66} in most of the blow-up ball.

As for the surface energy concentrated on $\partial^* \mathcal{P}$, we blow-up around $x_0 \in \BBB\partial^* \mathcal{P}\cap \Omega\EEE$ to find that the rescaled function converges in measure, up to subtracting  in the two halves of the blow-up cell two different infinitesimal rigid motions whose difference diverges as $h\to +\infty$, so that 
 the jump has  arbitrarily large amplitude 
 near the middle of the cell. This allows to conclude through a slicing argument (anisotropic), \BBB which requires a suitable condition on $g_\infty$, such as 
 ($g_5$) (notice that this condition, which states that
$g_\infty$ is independent from the amplitude of the jump, is very
restrictive, however we have currently no idea of how to treat more
general cases).
  \EEE

Theorem~\ref{thm:minimization} ensures existence of solutions to
the class of minimisation problems
\begin{equation*}
\min_{(u, \mathcal{P})}  \bigg\{ \int_\Omega f(x,e(u))\dx + \int_{J_{u}\sm \partial^*\mathcal{P}} g( x \EEE, [u], \nu_u) \dh + \int_{\BBB\partial^* \mathcal{P}\cap \Omega\EEE} g_\infty( x, \EEE \nu) \dh    \bigg\}
\end{equation*}
under Dirichlet boundary condition. \BBB The further condition $\partial^* \mathcal{P}\cap \Omega\subset J_u$ may be enforced, permitting to detect the effective fractured zone by looking only at $u$. \EEE In this class of problems we minimise not only in $u$, but also on the possible partitions that may be created.
If $g_\infty\equiv +\infty$, minimising sequences converge without modifications, see Proposition~\ref{prop:minimizzazioneNormaInfinita}. The case $g(x,[u],\nu)=g_\infty(\nu)=\psi(\nu)$ corresponds to minimise an anisotropic version of \eqref{1805200928} with general nonhomogeneous bulk energy (see Proposition~\ref{prop:minimizzazioneNormaFinita}; we refer to \cite[Theorem~5.1]{CriFri20} for an anisotropic version of \eqref{1805200928} in the context of epitaxially strained materials \cite{BonCha02}).

The paper is organised as follows: in Section~\ref{Sec1} we recall basic notions and prove two lemmas on infinitesimal rigid motions. Section~\ref{Sec2} is devoted to the proof of Theorem~\ref{thm:main}. In Section~\ref{Sec3} we prove Theorem~\ref{thm:minimization} and address the Dirichlet minimisation problems.

\section{Preliminaries}\label{Sec1}
In this section we fix the notation and recall the main tools employed in this work.
\subsection{Basic notation} 
For every $x\in \Rn$ and $\varrho>0$, let $B_\varrho(x) \subset \Rn$ be the open ball with center $x$ and radius $\varrho$, and let $Q_\varrho(x) = x+(-\varrho, \varrho)^n$, $Q_\varrho^\pm(x) = Q_\varrho(x) \cap \{x \in \Rn \colon \pm x_1 >0\}$. For $\nu \in \Sn:=\{x \in \Rn \colon |x|=1\}$, we let also $Q_\varrho^\nu(x)$ the cube with ``center'' $x$, sidelength $\varrho$ and with a face in a plane orthogonal to $\nu$. We omit to write the dependence on $x$ when $x=0$.  (For $x$, $y\in \Rn$, we use the notation $x\cdot y$ for the scalar product and $|x|$ for the  Euclidean  norm.)   By ${\mathbb{M}^{n\times n}}$, ${\mathbb{M}^{n\times n}_{\rm sym}}$, and ${\mathbb{M}^{n\times n}_{\rm skew}}$ we denote the set of $n\times n$ matrices, symmetric matrices, and skew-symmetric matrices, respectively.   We write $\chi_E$ for the indicator function of any $E\subset \R^n$, which is 1 on $E$ and 0 otherwise.  If $E$ is a set of finite perimeter, we denote its essential boundary by $\partial^* E$, and by \BBB $E^{(s)}$ the set of points with density $s$ for $E$,  \EEE  see \cite[Definition 3.60]{AFP}.   
 We indicate the minimum  and maximum  value between $a, b \in \R$ by  $a \wedge b$  and $a \vee b$, respectively. 

We denote by $\Ln$ and $\mathcal{H}^k$ the $n$-dimensional Lebesgue measure and the $k$-dimensional Hausdorff measure, respectively. The $m$-dimensional Lebesgue measure of the unit ball in $\R^m$ is indicated by $\gamma_m$ for every $m \in \N$. For any locally compact subset $B  \subset \Rn$, (i.e.\ any point in $B$ has a neighborhood contained in a compact subset of $B$),
the space of bounded $\R^m$-valued Radon measures on $B$ [respectively, the space of $\R^m$-valued Radon measures on $B$] is denoted by $\mathcal{M}_b(B;\R^m)$ [resp., by $\mathcal{M}(B;\R^m)$]. If $m=1$, we write $\mathcal{M}_b(B)$ for $\mathcal{M}_b(B;\R)$, $\mathcal{M}(B)$ for $\mathcal{M}(B;\R)$, and $\mathcal{M}^+_b(B)$ for the subspace of positive measures of $\mathcal{M}_b(B)$. For every $\mu \in \mathcal{M}_b(B;\R^m)$, its total variation is denoted by $|\mu|(B)$.  Given $\Omega \subset \Rn$ open, we   use the notation  
$L^0(\Omega;\R^m)$  for  the space of $\Ln$-measurable functions $v \colon \Omega \to \R^m$, endowed with the topology of convergence in measure.

\begin{definition}
Let $E\subset \Rn$,  $v \in L^0(E;\R^m)$,   and  $x\in \Rn$ such that
\begin{equation*}
\limsup_{\varrho\to 0^+}\frac{\Ln(E\cap B_\varrho(x))}{\varrho^{n}}>0\,.
\end{equation*}
A vector $a\in \R^m$ is the \emph{approximate limit} of $v$ as $y$ tends to $x$ if for every $\varepsilon>0$ there holds
\begin{equation*}
\lim_{\varrho \to 0^+}\frac{\Ln(E \cap B_\varrho(x)\cap \{|v-a|>\varepsilon\})}{ \varrho^n  }=0\,,
\end{equation*}
and then we write
\begin{equation*}
\aplim \limits_{y\to x} v(y)=a\,.
\end{equation*}
\end{definition}

\begin{definition}
Let $U\subset \Rn$  be  open  and $v \in L^0(U;\R^m)$.    The \emph{approximate jump set} $J_v$ is the set of points $x\in U$ for which there exist $a$, $b\in \R^m$, with $a \neq b$, and $\nu\in \Sn$ such that
\begin{equation*}
\aplim\limits_{(y-x)\cdot \nu>0,\, y \to x} v(y)=a\quad\text{and}\quad \aplim\limits_{(y-x)\cdot \nu<0, \, y \to x} v(y)=b\,.
\end{equation*}
The triplet $(a,b,\nu)$ is uniquely determined up to a permutation of $(a,b)$ and a change of sign of $\nu$, and is denoted by $(v^+(x), v^-(x), \nu_v(x))$. The jump of $v$ is the function 
defined by $[v](x):=v^+(x)-v^-(x)$ for every $x\in J_v$. 
\end{definition}
 We note that  $J_v$ is a Borel set  with $\Ln(J_v)=0$, and that $[v]$ is a Borel function. 

\subsection{$BV$ and $BD$ functions}
 Let $U\subset \Rn$  be open. We say that a function $v\in L^1(U)$ is a \emph{function of bounded variation} on $U$, and we write $v\in BV(U)$, if $\mathrm{D}_i v\in \mathcal{M}_b(U)$ for  $i=1,\dots,n$,  where $\mathrm{D}v=(\mathrm{D}_1 v,\dots, \mathrm{D}_n v)$ is its distributional  derivative.  A vector-valued function $v\colon U\to \R^m$ is in $BV(U;\R^m)$ if $v_j\in BV(U)$ for every $j=1,\dots, m$.
The space $BV_{\mathrm{loc}}(U)$ is the space of $v\in L^1_{\mathrm{loc}}(U)$ such that $\mathrm{D}_i v\in \mathcal{M}(U)$ for $i=1,\dots,d$. 
If $n=1$, $v \in L^1(U)$ is a function of bounded variation if and only if its pointwise variation is finite, cf.\ \cite[Proposition~3.6, Definition~3.26, Theorem~3.27]{AFP}. 

A function $v\in L^1(U;\Rn)$ belongs to the space of \emph{functions of bounded deformation} if 
 the distribution 
$\mathrm{E}v := \frac{1}{2}((\mathrm{D}v)^T + \mathrm{D}v )$  belongs to $\mathcal{M}_b(U;\Mnn)$.
It is well known (see \cite{AmbCosDM97, Tem}) that for $v\in BD(U)$, $J_v$ is countably $(\hn, n{-}1)$ rectifiable, and that
\begin{equation*}
\mathrm{E}v=\mathrm{E}^a v+ \mathrm{E}^c v + \mathrm{E}^j v\,,
\end{equation*}
where $\mathrm{E}^a v$ is absolutely continuous \BBB w.r.t.\ \EEE $\Ln$, $\mathrm{E}^c v$ is singular \BBB w.r.t.\ \EEE $\Ln$ and such that $|\mathrm{E}^c v|(B)=0$ if $\hn(B)<\infty$, while $\mathrm{E}^j v$ is concentrated on $J_v$. The density of $\mathrm{E}^a v$ \BBB w.r.t.\ \EEE $\Ln$ is denoted by $e(v)$.

The space $SBD(U)$ is the subspace of all functions $v\in BD(U)$ such that $\mathrm{E}^c v=0$. For $p\in (1,\infty)$, we define
$SBD^p(U):=\{v\in SBD(U)\colon e(v)\in L^p(\Omega;\Mnn),\, \hn(J_v)<\infty\}$.
Analogous properties hold for $BV$,  such as  the countable rectifiability of the jump set and the decomposition of $\mathrm{D}v$.  The  spaces $SBV(U;\R^m)$ and $SBV^p(U;\R^m)$ are defined similarly, with $\nabla v$, the density of $\mathrm{D}^a v$, in place of $e(v)$.
For a complete treatment of $BV$, $SBV$ functions and $BD$, $SBD$ functions, we refer to \cite{AFP} and to \cite{Tem, AmbCosDM97, BelCosDM98, ConFocIur17ProcEdi}, respectively.

\subsection{$GBD$ functions}
The space $GBD$ of \emph{generalized functions of bounded deformation} has been introduced in \cite{DM13}. We recall its definition and main properties, referring to that paper for a general treatment and more details. Since the definition of $GBD$ is given by slicing (differently from the definition of $GBV$, cf.~\cite{Amb90GSBV, DeGioAmb88GBV}), we  first introduce some notation. Fixed $\xi \in \Sn$, we let
\begin{equation}\label{eq: vxiy2}
\Pi^\xi:=\{y\in \Rn\colon y\cdot \xi=0\},\qquad B^\xi_y:=\{t\in \R\colon y+t\xi \in B\} \ \ \ \text{ for any $y\in \Rn$ and $B\subset \Rn$}\,,
\end{equation}
and for every function $v\colon B\to  \Rn$ and $t\in B^\xi_y$, let
\begin{equation}\label{eq: vxiy}
v^\xi_y(t):=v(y+t\xi),\qquad \widehat{v}^\xi_y(t):=v^\xi_y(t)\cdot \xi\,.
\end{equation}

\begin{definition}[\cite{DM13}]
Let $\Omega\subset \Rn$ be a  bounded open set, and let  $v \in L^0(\Omega;\Rn)$.   Then $v\in GBD(\Omega)$ if there exists $\lambda_v\in \mathcal{M}^+_b(\Omega)$ such that  one of the following equivalent conditions holds true 
 for every 
$\xi \in \Sn$: 
\begin{itemize}
\item[(a)] for every $\tau \in C^1(\R)$ with $-\tfrac{1}{2}\leq \tau \leq \tfrac{1}{2}$ and $0\leq \tau'\leq 1$, the partial derivative $\mathrm{D}_\xi\big(\tau(v\cdot \xi)\big)=\mathrm{D}\big(\tau(v\cdot \xi)\big)\cdot \xi$ belongs to $\mathcal{M}_b(\Omega)$, and for every Borel set $B\subset \Omega$ 
\begin{equation*}
\big|\mathrm{D}_\xi\big(\tau(v\cdot \xi)\big)\big|(B)\leq \lambda_v(B);
\end{equation*}
\item[(b)] $\widehat{v}^\xi_y \in BV_{\mathrm{loc}}(\Omega^\xi_y)$ for $\hn$-a.e.\ $y\in \Pi^\xi$, and for every Borel set $B\subset \Omega$ 
\begin{equation*}
\int_{\Pi^\xi} \Big(\big|\mathrm{D} {\widehat{v}}_y^\xi\big|\big(B^\xi_y\setminus J^1_{{\widehat{v}}^\xi_y}\big)+ \mathcal{H}^0\big(B^\xi_y\cap J^1_{{\widehat{v}}^\xi_y}\big)\Big)\dh(y)\leq \lambda_v(B)\,,
\end{equation*}
where
$J^1_{{\widehat{u}}^\xi_y}:=\left\{t\in J_{{\widehat{u}}^\xi_y} : |[{\widehat{u}}_y^\xi]|(t) \geq 1\right\}$.
\end{itemize} 
The function $v$ belongs to $GSBD(\Omega)$ if $v\in GBD(\Omega)$ and $\widehat{v}^\xi_y \in SBV_{\mathrm{loc}}(\Omega^\xi_y)$ for 
every
$\xi \in \Sn$ and for $\hn$-a.e.\ $y\in \Pi^\xi$.
\end{definition} 
Every $v\in GBD(\Omega)$ has an \emph{approximate symmetric gradient} $e(v)\in L^1(\Omega;\Mnn)$ such that for every $\xi \in \Sn$ and $\hn$-a.e.\ $y\in\Pi^\xi$ there holds
\begin{equation}\label{3105171927}
 e(v)(y +  t\xi)   \xi   \cdot \xi= (\widehat{v}^\xi_y)'(t)  \quad\text{for } \mathcal{L}^1\text{-a.e.\ }   t \in  \Omega^\xi_y\,;
\end{equation}
the \emph{approximate jump set} $J_v$ is still countably $(\hn,n{-}1)$-rectifiable (\textit{cf.}~\cite[Theorem~6.2]{DM13}) and may be reconstructed from its slices through the identity
\begin{equation}\label{slicingSaltoGSBD}
(J^\xi_v)^\xi_y= J_{\widehat{v}^\xi_y} \quad\text{and}\quad v^\pm(y+t\xi)  \cdot \xi = (\widehat{v}^\xi_y)^\pm(t) \ \text{ for }t\in (J_v)^\xi_y\,,
\end{equation}
where $J^\xi_v:=\{x\in J_v \colon [v]\cdot \xi \neq 0\}$ (it holds that $\hn(J_v\sm J_v^\xi)=0$ for $\hn$-a.e.\ $\xi \in \Sn$).
It follows that, if 
$v \in GSBD(\Omega)$ with $\hn(J_v)<+\infty$, for every Borel set $B\subset \Omega$
\begin{equation}\label{0101182248}
\hn(J_v\cap B)=(2 \gamma_{n-1})^{-1} \int\limits_{\Sn} \bigg(\int \limits_{\Pi^\xi} \ho(J_{v\xy}\cap B\xy)\,\mathrm{d}\hn(y)\bigg)\,\mathrm{d}\hn(\xi)
\end{equation}
and the two conditions in the definition of $GSBD$ for $v$ hold for $\lambda_v \in \mathcal{M}_b^+(\Omega)$ such that
\begin{equation}\label{0201181308}
\lambda_v(B)\leq \int_B |e(v)| \dx + \hn(J_v \cap B)\quad\text{for every Borel set $B\subset \Omega$}\,.
\end{equation} 

 For any countably $(\hn,n{-}1)$-rectifiable set $M \subset \Omega$ with unit normal $\nu \colon M \to \Sn$, it holds that for $\hn$-a.e.\ $x\in M$ there exist the \emph{traces} $v_M^+(x)$, $v_M^-(x)\in \Rn$ such that
\begin{equation}\label{0106172148}
\aplim \limits_{\pm(y-x)\cdot \nu(x)>0, \, y\to x} \hspace{-1em} v(y)=v_M^{\pm}(x)
\end{equation}
and they can be reconstructed from the traces of the one-dimensional slices. This has been proven by \cite[Theorem~5.2]{DM13} for $C^1$ manifolds of dimension $n{-}1$, and may be extended to countably $(\hn,n{-}1)$-rectifiable sets arguing as in \cite[Proposition~4.1, Step~2]{Bab15}.


 Finally, if $\Omega$ has Lipschitz boundary, for each $v\in GBD(\Omega)$ the traces on $\partial \Omega$ are well defined in the sense that for $\mathcal{H}^{n-1}$-a.e.\ $x \in \partial\Omega$ there exists ${\rm tr}(v)(x) \in \R^n$ such that
$$\aplim\limits_{y \to x, \ y \in \Omega} v(y) = {\rm tr}(v)(x). $$
 For $1 < p < \infty$, the   space $GSBD^p(\Omega)$ is  defined by  
\begin{equation*}
GSBD^p(\Omega):=\{v\in GSBD(\Omega)\colon e(v)\in L^p(\Omega;\Mnn),\, \hn(J_v)<\infty\}\,.
\end{equation*}
We say that a sequence $(v_k)_k \subset GSBD^p(\Omega)$ \emph{converges weakly} to $v \in GSBD^p(\Omega)$ if 
\begin{align}\label{eq: weak gsbd convergence}
 \sup_{k\in \N} \big( \Vert e(u_k) \Vert_{L^p(\Omega)} + \mathcal{H}^{n-1}(J_{u_k})\big) < + \infty \ \ \ \text{and} \ \ \    u_k \to u  \text{ in } L^0(\Omega;\Rn)\,.
 \end{align}
 We say that $(v_k)_k$ \emph{is bounded in } $GSBD^p(\Omega)$ if $\sup_{k} \big( \Vert e(u_k) \Vert_{L^p(\Omega)} + \mathcal{H}^{n-1}(J_{u_k})\big) < + \infty$.

We recall the following approximate Korn-type inequality for $GSBD^p$ functions with small jump set in a ball, recently proven in \cite{CagChaSca20}.  (We fix the case $\varepsilon=1$ in that result.) \EEE
 We refer to \cite{FriPWKorn} and \cite{CFI16ARMA} for Korn-type inequalities in $GSBD^p$ in two dimensions. 
\begin{theorem}[\cite{CagChaSca20}, Theorem~3.2]\label{th:KornGSBD}
Let $n \in \N$ with $n\geq 2$, and let $p\in (1,+\infty)$. Given $\sigma \in (0,1)$ there exist $C=C(n,p)$
and $\eta=\eta(n,p,\sigma)$, 
such that for every $\varrho>0$, $v\in GSBD^p(B_\varrho)$ with $\hn(J_v)\leq \eta \varrho^{n-1}$ there exist $w \in GSBD^p(B_\varrho)$ and a set of finite perimeter $\omega\subset B_\varrho$ such that $w=v$ in $B_\varrho\sm \omega$, $\hn(\partial^* \omega)<C \hn(J_v)$, $w \in W^{1,p}(B_{\sigma \varrho};\Rn)$, and
\begin{equation*}
\int_{B_\varrho} |e(w)|^p \dx \leq 2 \int_{B_\varrho} |e(v)|^p \dx\,,\qquad \hn(J_w) \leq \hn(J_v)\,.
\end{equation*}
\end{theorem}
Employing this result, in \cite{CagChaSca20} it is proven that any function $v\in GSBD^p(\Omega)$ is approximately differentiable $\Ln$-a.e.\ in $\Omega$, that is for $\Ln$-a.e.\ $x \in \Omega$ there exists $\nabla v(x) \in \mathbb{M}^{n\times n}$ (such that $e(v)(x)= (\nabla v(x))^{\mathrm{sym}}$ for a.e. $x$) for which it holds
\begin{equation*}
 \aplim_{y\to x}  \frac{|v(y)-v(x)- \nabla v(x) (y-x)|}{|y-x|} =0\,. 
\end{equation*}
 \subsection{Caccioppoli partitions.} A partition $\mathcal{P}=(P_j)_{j}$ of an open set $U\subset \Rn$ is said a \emph{Caccioppoli partition} of $U$ if $\sum_{j\in \N}\partial^* P_j < +\infty$. (see \cite[Definition~4.16]{AFP}). 
For Caccioppoli partitions the following structure theorem holds. 
\begin{theorem}[\cite{AFP}, Theorem~4.17]\label{thm:Caccioppoli}
Let $(P_j)_{j}$ be a Caccioppoli partition of $U$. Then
\begin{equation*}
\bigcup_{j\in \N} P_j^{(1)} \cup \bigcup_{i\neq j} (\partial^* P_i \cap \partial^* P_j)
\end{equation*}
contains $\hn$-almost all of $U$.
\end{theorem} 
For any Caccioppoli partition $\mathcal{P}=(P_j)_{j}$ we  set
\begin{equation}\label{1305201259}
\partial^* \mathcal{P}:= \bigcup_{j \in \N} \partial^* P_j,\qquad \mathcal{P}^{(1)}:=\bigcup_{j\in \N} P_j^{(1)},\qquad \BBB \nu_{\mathcal{P}}(x):=\nu_{P_j}(x) \quad\text{for }x\in \partial^* P_j \sm \bigcup_{i<j}\partial^* P_i\,. \EEE
\end{equation}
\EEE
\subsection{Symmetric quasi-convexity.} We recall the definition of symmetric quasi-convex functions, introduced in \cite{Ebo01}.
\begin{definition}[\cite{Ebo01}]\label{def:symquasiconvexity}
A function $f\colon \Mnn \to [0,+\infty)$ is \emph{symmetric quasi-convex} if
\begin{equation*}
f(\xi) \leq \frac{1}{\Ln(D)}\int_D f(\xi + e(\varphi)(x)) \dx
\end{equation*}
for every bounded open set $D$ of $\Rn$, for every $\varphi\in W^{1,\infty}_0(D;\Rn)$, and for every $\xi \in \Mnn$.
\end{definition}
This property is related to the \emph{quasi-convexity} in the sense of Morrey \cite{Mor66}; indeed $f$ is symmetric quasi-convex if and only if $f\circ \pi$ is quasi-convex in the sense of Morrey, where $\pi$ denotes the projection of $\mathbb{M}^{n{\times}n}$ onto $\Mnn$ (see \cite[Remark~2.3]{Ebo01}).

 
 \subsection{Some lemmas on affine functions.} We present below \BBB three \EEE lemmas that will be useful in the slicing procedure in Theorems~\ref{thm:main} and \ref{thm:minimization}. 
  We say that $a\colon \R^d \to \R^d$ is an \emph{infinitesimal rigid motion} if $a$ is affine with $ e(a) = \frac{1}{2}  ( \nabla a + (\nabla a)^{\mathrm{T}})  = 0$.   \EEE
 \begin{lemma}\label{le:2502200924}
Let $(a_h)_h$ be a sequence of \BBB piecewise \EEE rigid motions such that \eqref{2302202219} and \eqref{2202201909} hold. Then 
for $\hn$-a.e.\ $\xi \in \Sn$
\begin{equation}\label{2502200925}
|(a_h^j-a_h^i)(x)\cdot \xi| \to +\infty \quad\text{as }h\to +\infty \quad \text{for }\Ln\text{-a.e.\ }x\in \Omega, \text{ for all }i\neq j\,,
\end{equation}
\end{lemma}
\begin{proof}
For fixed $i \in \N$, $j \in \N$, with $i \neq j$, \eqref{2502200925} follows from \cite[Lemma~2.7]{CC18} applied to $v_h=a_h^j-a_h^i$. This provides an $\hn$-negligible set of $\xi$ $N_{i,j}\subset \Sn$. 

Then \eqref{2502200925} hold for any $i\neq j$ for every $\xi \in \Sn \sm N$, where $N=\bigcup_{i\neq j} N_{i,j}$ is still $\hn$-negligible.
\end{proof}

\begin{lemma}\label{le:2503201833}
Let $F \subset \Omega$ be such that $\hn(F)<+\infty$,
and let $(a_h)_h$ be a sequence of \BBB piecewise \EEE rigid motions such that \eqref{2302202219} and \eqref{2202201909} hold. Then there exist \GGG  a subsequence, still denoted by $(a_h)_h$, and \EEE two sets $D \subset \Sn$, countable and dense in $\Sn$, and $N \subset F$ with $\hn(N)=0$, such that
\begin{equation}\label{2503201839}
|(a^i_h-a^j_h)(x) \cdot \xi| \to +\infty \quad\text{as }h\to +\infty\quad\text{for every }\xi \in D,\, x \in F \sm N,\, i\neq j\,.
\end{equation}
\end{lemma}
\begin{proof}
For any $i \neq j$, and $h \in \N$, we set $(a^i_h-a^j_h)(x)=A^{i,j}_h\, x + b^{i,j}_h$, with $A^{i,j}_h \in \mathbb{M}^{n\times n}_{\rm skew}$ 
and $b^{i,j}_h \in \Rn$. 
Then 
\begin{equation}\label{3107212217}
(a^i_h-a^j_h)(x) \cdot \xi= - A^{i,j}_h \xi \cdot x + b^{i,j}_h\cdot \xi\,.
\end{equation}
\GGG
For $i\neq  j$, let $\mu_h^{i,j}:=|A_h^{i,j}|+|b^h_{i,j}|$; it holds that $\mu_h^{i,j}\to +\infty$ as $h\to +\infty$.
Then, we can extract iteratively a subsequence (not relabelled)
such that for each pair $(i,j)$, $i\neq j$ 
\begin{equation}\label{2809211651}
\frac{A_h^{i,j}}{\mu_h^{i,j}}\to A^{i,j}\quad\text{and}\quad \frac{b_h^{i,j}}{\mu_h^{i,j}}\to b^{i,j}\quad\text{ with $|A^{i,j}|+|b^{i,j}|=1$.}
\end{equation}


We observe that, since $\hn(F)<+\infty$, there exist at most countably many hyperplanes $(H_l)_{l\in \N}$ such that $\hn(F \cap H_l)>0$.
We then define
\begin{equation*}
D\text{ a countable dense subset of } \widehat{D}:=\Sn \sm \left(E \cup \bigcup_{i\neq j: A^{i,j}= 0} (b^{i,j})^\perp \right),
\end{equation*}
for
\[
 E:= \bigcup_{i\neq j: A^{i,j}\neq 0} \left\{ \mathrm{ker}\, A^{i,j}\cup \bigcup_{l\in \N}  \{ \xi \in \Sn \colon A^{i,j}\xi \text{ is orthogonal to } H_l \} \right\}.
\]
Such a set $D$ exists because
the set $\widehat{D}$ has full measure in $\Sn$:
indeed $\Sn\setminus \widehat{D}$ is included in a countable union
of linear subspaces of dimension at most $n{-}1$ (as $A^{i,j}$ has rank at least
two, or observing that if $A^{i,j}\xi$ is orthogonal to $H_l$
then $\xi$ must be parallel to $H_l$ since $(A^{i,j}\xi)\cdot\xi=0$).

%
%

In case $A^{i,j}=0$, for any $\xi\in D$ and $x \in \Omega$ one has $|(a^i_h(x)-a^j_h(x))\cdot\xi|\to+\infty$. In fact,
$ b^{i,j}\cdot \xi\neq 0$ by definition of $D$, and 
\[
\lim_{h\to +\infty}(a^i_h(x)-a^j_h(x))\cdot\xi =\lim_{h\to +\infty}(\mu^{i,j}_h  b^{i,j} \cdot \xi)=+\infty
\]
for every $x \in \Omega$,  by \eqref{2809211651}.

Now we consider $(i,j)$, $i\neq j$, with $A^{i,j}\neq 0$. We claim that given $\xi\in D$, there is a $\hn$-negligible set $N_\xi^{i,j}$ such that $|(a^i_h(x)-a^j_h(x))\cdot\xi|\to+\infty$ for every $x \in F \sm N_\xi^{i,j}$. In fact, as soon as $-A^{i,j}\xi\cdot x + b^{i,j}\cdot \xi\neq 0$,
we have that $|(a^i_h(x)-a^j_h(x))\cdot\xi|\to+\infty$, by \eqref{2809211651}; on the other hand, $H_\xi^{i,j}:=\{x\in \Omega\colon A^{i,j}\xi \cdot x = b^{i,j}\cdot \xi\}$ is an hyperplane (since $\xi \notin \mathrm{ker}\,A^{i,j}$) such that
\begin{equation}\label{eq:nonzero}
 \hn(F \cap H_\xi^{i,j})=0,
\end{equation}
indeed $H_\xi^{i,j}$ is different from all the hyperplanes $H_l$
by the choice of $\xi \not\in E$.
Letting $N_\xi^{i,j}:=F \cap H_\xi^{i,j}$ and 
\[
N:= \bigcup_{\xi \in D} \, \bigcup_{i\neq j: A^{i,j}\neq 0} N_\xi^{i,j},
\]
we find that $\hn(N)=0$ and~\eqref{2503201839} holds. 
\end{proof}

\GGG \begin{lemma}\label{le:0108211842}
Let $(a_k)_k$ be a sequence of infinitesimal rigid motions, $a_k\colon \Omega\to \Rn$, $a_k(x)=A_k x +b_k$, $A_k\in \Mskew$, $b_k\in \Rn$.
Then, up to a subsequence, either $(a_k)_k$ converges uniformly to an infinitesimal rigid motion  
(with values in $\Rn$) or there exists an affine subspace $\Pi$ of dimension at most $n{-}2$ such that $|a_k(x)|\to +\infty$ for every $x$ in $\Omega \sm \Pi$.
\end{lemma}
\begin{proof}
If $A_k$ and $b_k$ are uniformly bounded on a subsequence, then we have uniform convergence to an infinitesimal rigid motion (with values in $\Rn$).
Else, $\mu_k:= |A_k|+|b_k|$ diverges.
We argue similarly to what done in the proof of Lemma~\ref{le:2503201833}.
Up to
a subsequence, we may assume that 
\begin{equation}\label{2909210951}
\frac{A_k}{\mu_k}\to A \ \text{ and } \
\frac{b_k}{\mu_k}\to b \quad \text{ with }|A|+|b|=1.
\end{equation}
Then, for any $x$,
\[
 Ax+b \neq 0\quad \Rightarrow \quad
  \lim_{k\to +\infty}|a_k(x)|=\lim_{k\to +\infty} \mu_k|Ax+b|=+\infty
\] 
If $A=0$, this is true for all $x$ (since $b\neq 0$). Else, this
is true as long as $x$ does not belong to the affine subspace $\Pi:=\{x\colon Ax+b=0\}$, which has dimension at most $n{-}2$, $A$ being a non null skew-symmetric matrix.  
%
%
%
%
%
\end{proof}
\EEE
\section{The compactness result}\label{Sec2}
This section is devoted to the proof of Theorem~\ref{thm:main}.
\begin{proof}[Proof of Theorem~\ref{thm:main}]
We divide the proof in steps.  \medskip
\paragraph*{\textbf{Step~1: Existence of $(a_h)_h$.}}
Let $\mu_h:=\hn \mres J_{u_h} \in \mathcal{M}_b^+(\Omega)$. Since $(u_h)_h$ is bounded in $GSBD^p(\Omega)$, $\sup_h |\mu_h|(\Omega)=\hn(J_{u_h})<M$ and then, up to a (not relabelled) subsequence, $\mu_h \wstar \mu $ in $\mathcal{M}_b^+(\Omega)$. We denote by
\begin{equation*}
J:=\Big\{x\in \Omega \colon \limsup_{\varrho\to 0^+}\frac{\mu(B_\varrho(x))}{\varrho^{n-1}}>0\Big\}\,.
\end{equation*}
By \cite[Theorem~2.56]{AFP}, the set $J$ is $\sigma$-finite \BBB w.r.t.\ \EEE the measure $\hn$, so in particular $\Ln(J)=0$.
Let us fix $\ol\sigma \in (\tfrac12, 1)$ and consider $\ol\eta=\ol\eta(\ol\sigma)$  and $C>0$ \EEE
such that the conclusion of Theorem~\ref{th:KornGSBD} holds true in correspondence to $\ol\sigma$.  (We assume $n$ and $p$ fixed once for all.) \EEE
\medskip
\paragraph*{\textit{Substep~1.1: Existence of $(a_h)_h$ up to a set of small measure.}}
Let us fix $\eta \in (0, \ol \eta)$. In 
 the following 
we perform a construction in correspondence of $\ol\sigma$ and $\eta$; to ease the notation, we do not write explicitly the dependence on these parameters in the objects introduced in the construction. Afterwards (starting from \eqref{2302201929}), we shall keep track of the dependence on $\eta$.
 
By definition of $J$, for any $x \in \Omega\sm J$ there exists $\varrho_0=\varrho_0(x,\eta)$ such that $\mu(B_\varrho(x))\leq \frac{\eta}{2}\varrho^{n-1}$ for every $\varrho\in (0, \varrho_0)$. Then, in view of the weak$^*$ convergence of $\mu_h$ to $\mu$, for every $\varrho \in (0, \varrho_0)$ such that $\mu(\partial B_\varrho(x))=0$ (notice that this holds for all $\varrho$ except countable many) we have that $\lim_{h\to \infty}\mu_h(B_\varrho(x))=\mu(B_\varrho(x))$. \BBB We denote by $T_0=T_0(x)$ the set of $\varrho<\varrho_0$ for which $\lim_{h\to \infty}\mu_h(B_\varrho(x))=\mu(B_\varrho(x))$. \EEE This implies that there exists \BBB $h_0=h_0(x,\eta,\varrho)$ \EEE such that $\mu_h(B_\varrho(x))< \eta \varrho^{n-1}$ \BBB for $\varrho\in T_0$ and $h\geq h_0$. \EEE

Applying Theorem~\ref{th:KornGSBD} in correspondence to $u_h \in GSBD^p(B_\varrho(x))$ we deduce that for every $x \in \Omega\sm J$, \BBB $\varrho\in T_0$, $h\geq h_0$, \EEE there exist $v_{h,\varrho,x} \in GSBD^p(B_\varrho(x))$ and a set of finite perimeter $\omega_{h,\varrho,x} \subset B_\varrho(x)$ such that $v_{h,\varrho,x}=u_h$ in $B_\varrho(x)\sm \omega_{h,\varrho,x}$, $v_{h,\varrho,x} \in W^{1,p}(B_{\ol \sigma \varrho}(x);\Rn)$, and
\begin{subequations}\label{eqs:2302201716}
\begin{align}
\int_{B_\varrho(x)} |e(v_{h,\varrho,x})|^p \dx & \leq  2 \EEE \int_{B_\varrho(x)} |e(u_h)|^p \dx\,,\label{2302201731}\\
 \hn(\partial^* \omega_{h,\varrho,x}) & \leq  C\, \EEE \hn(J_{u_h} \cap B_\varrho(x))\,,\label{2302201732}\\
\Ln(\omega_{h,\varrho,x}) & \leq  C\, \EEE \eta^{\frac{n}{n-1}} \varrho^n \,,\label{2302201733}
\end{align}
\end{subequations}
\BBB where in \eqref{2302201733} we used also the Isoperimetric Inequality. \EEE
In particular, by Korn and Korn-Poincaré inequalities applied to $v_{h,\varrho,x}$ in $B_{\ol\sigma \varrho}(x)$, there exist infinitesimal rigid motions $a_{h,\varrho,x}$ such that
\begin{equation}\label{2302200929}
\int_{B_{\ol\sigma \varrho}(x)} \Big(| v_{h,\varrho,x} -a_{h,\varrho,x}|^{p^*} + |\nabla (v_{h,\varrho,x} -a_{h,\varrho,x})|^p \Big) \dx \leq C  \int_{B_\varrho(x)} |e(u_h)|^p \dx\,.
\end{equation}
We notice that the family
\[
\mathcal{F}:=\{ B_{\ol \sigma \varrho}(x) \colon x \in \Omega\sm J,\, \BBB \varrho \in T_0(x)\EEE\}
\] is a fine cover of $\Omega\sm J$ (cf.\ \cite[Section~2.4]{AFP}). Then, by \BBB Besicovitch \EEE Covering Theorem, there exists a disjoint family of balls $(B_{\ol\sigma\varrho(x_i)}(x_i))_i \subset \mathcal{F}$ with $\Ln\Big((\Omega\sm J) \sm \bigcup_{i\in \N}  B_{\ol\sigma\varrho(x_i)}(x_i) \Big)=0$. In particular, there exists $N$, depending on $\eta$, such that
\begin{equation}\label{3103201628}
\Ln\Big((\Omega\sm J) \sm \bigcup_{i=1}^N  B_{\ol\sigma\varrho(x_i)}(x_i) \Big)<\eta\,.
\end{equation}
Let us fix $i \in \{1,\dots, N\}$. 
There exist (we set $\varrho_i \equiv \varrho(x_i)$) sequences of sets of finite perimeter $(\omega_{h, \varrho_i, x_i})_h$ contained in $B_{\varrho_i}(x_i),$ of functions $(v_{h,\varrho_i, x_i})_h \subset GSBD^p(B_{\varrho_i}(x_i)) \cap W^{1,p}(B_{\ol\sigma \varrho_i}(x_i);\Rn)$  with $v_{h,\varrho_i, x_i}=u_h$ in  $B_{\varrho_i}(x_i)\sm \omega_{h, \varrho_i, x_i}$, and of infinitesimal rigid motions $(a_{h,\varrho_i, x_i})_h$ such that \eqref{eqs:2302201716} and \eqref{2302200929} hold for $\varrho=\varrho_i$, $x=x_i$.

Then, by \eqref{2302201732} and \eqref{2302201733}, up to a subsequence (not relabelled) the characteristic functions of the sets $\omega_{h,\varrho_i, x_i}$ converge weakly$^*$ in $BV(B_{\varrho_i}(x_i))$ as $h\to +\infty$ to a set $\omega_{\varrho_i, x_i} \subset B_{\varrho_i}(x_i)$ with
\begin{equation}\label{2302201934}
\Ln(\omega_{\varrho_i,x_i})  \leq  C\, \EEE \eta^{\frac{n}{n-1}} {\varrho_i}^n\,.
\end{equation}
Moreover, again up to a (not relabelled) subsequence, 
\begin{equation}\label{2302202036}
v_{h, \varrho_i, x_i} - a_{h, \varrho_i, x_i} \weak u_{\varrho_i, x_i} \quad \text{in }W^{1,p}(B_{\ol \sigma \varrho_i}(x_i); \Rn)\,.
\end{equation}
We may assume that the convergences above hold along the same subsequence, independently on $i \in \{1,\dots, N\}$.
Let us denote
\begin{equation}\label{2302201929}
\omega^\eta:= \bigcup_{i =1}^N \big( \omega_{\varrho_i, x_i} \cap B_{\ol\sigma \varrho_i}(x_i) \big)\,,\quad a^\eta_h:=\sum_{i =1}^N a_{h,\varrho_i, x_i} \chi_{B_{\ol\sigma \varrho_i}(x_i)}\,,\quad u^\eta:=\sum_{i =1}^N u_{\varrho_i, x_i} \chi_{B_{\ol\sigma \varrho_i}(x_i)}\,.
\end{equation}
By \eqref{2302201934} we get
\begin{subequations}\label{eqs:2302202043}
\begin{equation}\label{2302202034}
\Ln(\omega^\eta) \leq C(\ol \sigma,n)  \, \eta^{\frac{n}{n-1}} \, \Ln(\Omega)\,,
\end{equation}
and \BBB \eqref{2302202036} implies \EEE that
\begin{equation}\label{2302202035}
u_h - a^\eta_h \to u^\eta\quad\text{in }L^p(\Omega\sm E^\eta; \Rn)\,,\qquad\text{for }E^\eta:=\omega^\eta \cup \Big((\Omega\sm J) \sm \bigcup_{i=1}^N  B_{\ol\sigma\varrho(x_i)}(x_i)\Big)\,.
\end{equation}
\end{subequations}
We may now find a partition $\mathcal{P}^\eta= (P_j^\eta)_j$ of $\Omega\sm E^\eta$ and a function $\tilde{u} \in L^p(\Omega\sm E^\eta;\Rn)$ such that, up to extracting a further subsequence \BBB w.r.t.\ \EEE $h$, in correspondence to any \BBB $P^\eta_j$ \EEE there is a sequence $(a_{h,j}^\eta)_h$ such that 
\begin{subequations}\label{eqs:2302202100}
\begin{equation}\label{2302202101}
|a_{h,j}^\eta(x) - a_{h,i}^\eta(x)| \to +\infty \text{ for a.e.\ } x \in \Omega \text{ whenever }i \neq j
\end{equation} 
and
\begin{equation}\label{2302202102}
u_h-a_{h,j}^\eta \to \tilde{u}^\eta \quad\text{in }L^{\BBB p \EEE}(P_j^\eta; \Rn)\,.
\end{equation}
\end{subequations}
In fact, this is done as follows by regrouping the sequences of infinitesimal rigid motions in each $B_{\ol\sigma \varrho_i}(x_i)$ in equivalence classes, up to extracting a further subsequence.

\BBB By Lemma~\ref{le:0108211842}, \EEE denoting $a_h^i \equiv a_{h,\varrho_i,x_i}$ for every $i \in \N$,  we may extract a subsequence (not relabelled) such that any sequence in 
\begin{equation*}
G:= \{(a^1_h)_h\} \cup \bigcup_{1\leq i < j \leq N} \{ (a^i_h-a^j_h)_h\}\,.
\end{equation*}
 either converges uniformly to an infinitesimal rigid motion
or diverges a.e.\ in $\Omega$.
We say that $i \neq j$ are in the same equivalence class if and only if $(a^i_h-a^j_h)_h$ converges uniformly to an infinitesimal rigid motion.
 
We conclude \eqref{eqs:2302202100} by considering the union of the $B_{\ol\sigma \varrho_i}(x_i) \sm \omega^\eta$ for the $i$'s in the same equivalence class, to get a partition $\mathcal{P}^\eta=(P_j^\eta)_j$, and by fixing a sequence of infinitesimal rigid motions as representative in each $P_j^\eta$. 
\medskip
\paragraph*{\textit{Substep~1.2: Conclusion of Step~1.}}
Let us now take the sequence $\eta_k:=\eta 2^{-k}$ for $k \in \N$. By a diagonal argument we may assume that \BBB \eqref{3103201628} and \EEE \eqref{eqs:2302202043} hold for the same subsequence $(u_h)_h$ for every $\eta_k$, for suitable $\omega^{\eta_k}$, $E^{\eta_k}$, $a_h^{\eta_k}$, $u^{\eta_k}$. Moreover, we find partitions $\mathcal{P}^{\eta_k}$ such that \eqref{eqs:2302202100} hold for $\eta_k$ in place of $\eta$.

Consider two $P^{\eta_{k_1}}_j$ and $P^{\eta_{k_2}}_i$ such that
\begin{equation*}
\Ln(P^{\eta_{k_1}}_j \cap P^{\eta_{k_2}}_i )>0\,.
\end{equation*} 
We notice that the sequences of infinitesimal rigid motions $(a_{h,j}^{\eta_{k_1}})_h$ and $(a_{h,i}^{\eta_{k_2}})_h$ belong to the same equivalence class, since 
$a_{h,j}^{\eta_{k_1}} - a_{h,i}^{\eta_{k_2}} \to \tilde{u}^{\eta_{k_1}}- \tilde{u}^{\eta_{k_2}}$ in $L^{\BBB p \EEE}(P^{\eta_{k_1}}_j \cap P^{\eta_{k_2}}_i ; \Rn)$.
This means that for any $k$, the partition $\mathcal{P}^{\eta_{k+1}}$ coincides with  the partition $\mathcal{P}^{\eta_{k}}$ in $\Omega\sm (E^{\eta_k} \cup E^{\eta_{k+1}})$.
Then the partition $\widetilde{\mathcal{P}}^{\eta_k}$ of $\Omega\sm \bigcap_{j\leq k} E^{\eta_j}$ characterised by
\begin{equation}\label{2508210729}
\widetilde{\mathcal{P}}^{\eta_k}=
\mathcal{P}^{\eta_j} \quad \text{in }\Omega\sm E^{\eta_j} \text{ for every }j\leq k
\end{equation}
is well defined, and such that the analogue of \eqref{2302202101} holds, that is for $\widetilde{P}^{\eta_k}_j$, $\widetilde{P}^{\eta_k}_i \in \widetilde{\mathcal{P}}^{\eta_k}$
\begin{equation}\label{2302202211}
|a_{h,j}^{\eta_k}(x) - a_{h,i}^{\eta_k}(x)| \to +\infty \text{ for a.e.\ } x \in B_1(0) \text{ whenever }i \neq j\,.
\end{equation} 
\BBB We remark that by \eqref{2508210729} we mean that for any $x \in \Omega\sm \bigcap_{j\leq k} E^{\eta_j}$ the set in the partition $\widetilde{\mathcal{P}}^{\eta_k}$ containing $x$ is the union of the sets in the partitions $\mathcal{P}^{\eta_j}$ containing $x$, as $j\leq k$. In this way we get that 
\begin{equation}\label{2508210803}
\widetilde{\mathcal{P}}^{\eta_k}=\widetilde{\mathcal{P}}^{\eta_k+1}\quad\text{ in }\Omega \sm \bigcap_{j\leq k} E^{\eta_j}\,.
\end{equation} 
Therefore $\widetilde{\mathcal{P}}^{\eta_k}$ are partitioning a larger set as $k$ increases, and any two of them coincide where regarded in common subdomains; moreover,  $(a_{h,j}^{\eta_k})_h$ and $(a_{h,l}^{\eta_k+1})_h$ are in the same equivalence class if $\Ln(\widetilde{\mathcal{P}}^{\eta_k}_j \cap \widetilde{\mathcal{P}}^{\eta_k+1}_l)>0$, arguing as above, so that we may assume that the sequences $(a_{h,j}^{\eta_k})_h$ and $(a_{h,l}^{\eta_k+1})_h$ coincide. This allows to find for any $K\in \N$ a piecewise rigid motion $a_h^K$ on $\Omega \sm \bigcap_{j\leq K} E^{\eta_j}$ such that \eqref{2302202211} holds for every $k\leq K$ and $a_h^k=a_h^{k+1}$ in $\Omega \sm \bigcap_{j\leq k} E^{\eta_j}$. \EEE

Since $\Ln(\bigcap_{j\leq k} E^{\eta_j})\leq \Ln(E^{\eta_k})\to 0$ as $k\to +\infty$ \BBB by \eqref{3103201628} and \eqref{2302202034}, 
 by the monotonicity property of partitions \eqref{2508210803} \EEE
 we obtain in the limit that there exists a partition $\mathcal{P}=(P_j)_j$ of $\Omega$ (up to a $\Ln$-negligible set) and piecewise 
 rigid motions defined as in \eqref{2302202219} such that \eqref{2202201909} holds and
\begin{equation}\label{2502201108}
u_h - a_h \to u \quad\text{in }L^0(\Omega;\Rn)\,,
\end{equation} 
\BBB where $a_h=a_h^k$ in $\Omega \sm \bigcap_{j\leq k} E^{\eta_j}$, for every $k\in \N$. 
The partition $\mathcal{P}$ is thus obtained by the monotonicity property \eqref{2508210803}; later we see that $\mathcal{P}$ is indeed a Caccioppoli partition.  
We remark that we do not have a uniform control of $u_h-a_h^k$ in $L^p$ norm w.r.t.\ $k$, while a pointwise convergence is guaranteed. \EEE
\medskip
\paragraph*{\textbf{Step~2: Proof of \eqref{eq:sciSalto}}}
In this step we follow a slicing strategy, in the spirit of \cite[Theorem~1.1]{CC18}. In particular, 
 the first part of the argument above is similar to that in \cite[Theorem~1.1, lower semicontinuity]{CC18}. We remark that we cannot directly apply that result (see Remark~\ref{rem:3103201937}).

We then fix $\xi \in \Sn$ in a set of full $\hn$-measure of $\Sn$ for which \eqref{2502200925} holds (cf.\ Lemma~\ref{le:2502200924}), and introduce
\begin{equation}\label{0101182132}
\I\xy(u_h):=\int\limits_\Oxy |(\dot{u}_h)\xy|^p\,\mathrm{d}t\,,
\end{equation}
where 
$(\dot{u}_h)\xy$ is the density of the absolutely continuous part of $\mathrm{D}(\widehat{u}_h)\xy$, the distributional derivative of $(\widehat{u}_h)\xy$ ($(\widehat{u}_h)\xy\in SBV^p_{\mathrm{loc}}(\Oxy)$ for every $\xi\in \Sn$ and for $\hn$-a.e.\ $y\in \Pi^\xi$, since $u_h\in GSBD^p(\Omega)$; we denote here and in the following $\widehat{u}_h$ for $\widehat{u_h}$).
Therefore
\begin{equation}\label{0101182137}
\int \limits_{\Pi_\xi} \I\xy(u_h) \,\mathrm{d}\hn(y)=\int \limits_\Omega |e(u_h)(x)\xi \cdot \xi  |^p \leq \int \limits_\Omega |e(u_h)|^p \dx \leq M\,,
\end{equation}
by Fubini-Tonelli's theorem and since $(u_h)_h$ is bounded in $GSBD^p(\Omega)$.
Let $u_k=u_{h_k}$ be a subsequence of $u_h$ such that
\begin{equation}\label{0101182244}
\lim_{k\to \infty} \hn(J_{u_k})=\liminf_{h\to \infty} \hn(J_{u_h})<+\infty\,,
\end{equation}
so that, 
by \eqref{0101182248}, \eqref{0101182137}, 
and Fatou's lemma, we have that for $\hn$-a.e.\ $\xi \in \Sn$
\begin{equation}\label{0101182302}
\liminf_{k\to\infty} \int\limits_{\Pi_\xi}\Big( \ho\big(J_{(\widehat{u}_k)\xy}\big) + \varepsilon \I\xy(u_k)  \Big)\,\mathrm{d}\hn(y) <+\infty\,,
\end{equation}
for a fixed $\varepsilon\in (0,1)$.
Let us fix $\xi \in \Sn$ such that \eqref{2502200925} and \eqref{0101182302} hold. Then there is a subsequence $u_m=u_{k_m}$ of $u_k$, depending on $\varepsilon$ and $\xi$, such that 
\begin{equation}\label{2502201156}
(\widehat{u}_m-\widehat{a}_m)\xy \to \widehat{u}\xy \quad\text{in }L^0(\Omega\xy) \quad\text{for $\hn$-a.e.\ $y\in \Pi_\xi$}
\end{equation}
and  
\begin{equation}\label{0101182324}
\begin{split}
\lim_{m\to\infty} &\int\limits_{\Pi_\xi}\Big( \ho\big(J_{(\widehat{u}_m)\xy}\big) + \varepsilon \I\xy(u_m) \Big)\,\mathrm{d}\hn(y)\\
&=\liminf_{k\to\infty} \int\limits_{\Pi_\xi}\Big( \ho\big(J_{(\widehat{u}_k)\xy}\big) + \varepsilon  \I\xy(u_k)  \Big)\,\mathrm{d}\hn(y)\,.
\end{split}
\end{equation}
As for \eqref{2502201156}, we notice that it follows from  Fubini-Tonelli's theorem and the convergence in measure of $u_h-a_h$ to $u$ (see \eqref{2502201108}), which corresponds to $\tanh(u_h-a_h) \to \tanh(u) \in L^1(\Omega; \Rn)$ (with $\tanh(v)=(\tanh (v\cdot e_1), \dots, \tanh (v \cdot e_n))$ for every $v \colon \Omega\to \Rn$). 
Therefore, by \eqref{0101182324} and Fatou's lemma, we have that for $\hn$-a.e.\ $y\in \Pi^\xi$
\begin{equation}\label{0101182328}
\liminf_{m\to \infty} \Big( \ho\big(J_{(\widehat{u}_m)\xy}\big) + \varepsilon \I\xy(u_m) \Big) < +\infty\,,
\end{equation}
Moreover, we infer that if $\xi$ satisfies \eqref{2502200925}, then 
\begin{equation}\label{2403201910}
\text{for $\hn$-a.e.\ $y\in \Pi^\xi$} \quad |(\widehat{a}_h^i-\widehat{a}_h^j)\xy(t)|= |(\widehat{a}_h^i-\widehat{a}_h^j)\xy(0)| \to +\infty \quad\text{for every }t\in \Omega\xy\,.
\end{equation}
Indeed, since $e(a^i_h)=0$ for every $i$, $h$, then, for fixed $h$, $(\widehat{a}_h^i-\widehat{a}_h^j)\xy$ is constant in $\Omega\xy$. Thus
\[
|(a^i_h-a^j_h)\cdot \xi|\to +\infty \ \text{ in } \ \bigcup \{ \Omega\xy \colon y\in \Pi^\xi \text{ s.t.\ } | (a^i_h-a^j_h)(y)\cdot \xi|\to +\infty\}\,,
\]
and \eqref{2403201910} holds true.

Let us consider $y\in \Pi^\xi$ satisfying \eqref{2502201156}, \eqref{0101182328},  \eqref{2403201910}, and such that 
$(\widehat{u}_m)\xy\in SBV_{\mathrm{loc}}(\Oxy)$
 for every $m$. Then we may extract a subsequence $u_j=u_{m_j}$ from $u_m$, depending also on $y$, for which
\begin{equation}\label{0101182346}
\lim_{j\to \infty} \Big( \ho\big(J_{(\widehat{u}_j)\xy}\big) + \varepsilon \I\xy(u_j)\Big)=\liminf_{m\to \infty} \Big( \ho\big(J_{(\widehat{u}_m)\xy}\big) + \varepsilon \I\xy(u_m) \Big)
\end{equation}
and
\begin{equation*}
(\widehat{u}_j-\widehat{a}_j)\xy \to \widehat{u}\xy \quad\mathcal{L}^1\text{-a.e.\ in }\Omega\xy\,,
\end{equation*}
\begin{equation}\label{2403201949}
 |( \widehat{a}^{i_1}_j-\widehat{a}^{i_2}_j)\xy(t)|=|( \widehat{a}^{i_1}_j-\widehat{a}^{i_2}_j)\xy(0)| \to +\infty \quad\text{for }t\in \Omega\xy \text{ and }i_1 \neq i_2\,.
\end{equation}
In the following, we denote  (similarly to \eqref{1305201259}, in dimension one) \EEE
\begin{equation*}
\partial \mathcal{P}\xy:= \bigcup_{j\in \N} \partial (P_j)\xy \BBB \cap \Omega\xy \EEE \subset \Omega\xy\,.
\end{equation*}

Since,  by \eqref{0101182346}, the number of jump points of $(\widehat{u}_j)\xy$ 
is  bounded uniformly \BBB w.r.t.\ \EEE $j$, up to pass to a subsequence of $((\widehat{u}_j)\xy)_j$ we may assume that for every $j$ 
\[
\ho\big(J_{(\widehat{u}_j)\xy}\big)=N_y \in \N\,.
\]
Therefore we have $M_y \leq N_y$ cluster points in the limit, denoted by 
\[t_1,\dots ,t_{M_y}\,.\]

Using the equiboundedness of $\mathrm{I}\xy(u_j)$, which follows from \eqref{0101182346},
we get that, for $\mathcal{L}^1$-almost any choice of $ \ol t \EEE \in (t_l, t_{l+1})$, 
\begin{equation}\label{2502201338}
 t \EEE \mapsto (\widehat{u}_j)\xy( t \EEE) - (\widehat{u}_j)\xy( \ol t \EEE) \text{ are equibounded \BBB w.r.t.\ \EEE $j$ in } W^{1,p}_\mathrm{loc}(t_l, t_{l+1})\,,
\end{equation} 
 by the Fundamental Theorem of Calculus,
 and then this sequence converges locally uniformly 
  in $(t_l, t_{l+1})$, as $j \to \infty$.
  
 Let us prove that
 \begin{equation}\label{2502201307}
 \partial \mathcal{P}\xy \subset \{ t_1,\dots ,t_{M_y} \}
\end{equation}  
We argue by contradiction, assuming that there exists $l\in \{1,\dots, M_y\}$ and $i_1$ such that $\partial (P_{i_1})\xy \cap (t_l, t_{l+1})\neq \emptyset$.
If this holds, there exist two sequences of infinitesimal rigid motions $(a_j^{i_1})_j$, $(a_j^{i_2})_j$ (the latter corresponds to some $P_{i_2}$ with $i_1 \neq i_2$)  such that
\begin{equation}\label{2502201322}
\begin{split}
&(\widehat{u}_j - \widehat{a}_j^{i_1})\xy \to \widehat{u}\xy \quad\mathcal{L}^1\text{-a.e.\ in } (P_{i_1})\xy \cap (t_l, t_{l+1}), \\
& (\widehat{u}_j - \widehat{a}_j^{i_2})\xy\to \widehat{u}\xy\quad\mathcal{L}^1\text{-a.e.\ in } (P_{i_2})\xy \cap (t_l, t_{l+1})\,,
\end{split}
\end{equation}
with $\mathcal{L}^1\big(  (P_{i_1})\xy  \cap (t_l, t_{l+1})  \big),\, \mathcal{L}^1\big(  (P_{i_2})\xy  \cap (t_l, t_{l+1})  \big)>0$. But this gives, with \eqref{2502201338} and since $\widehat{a}_j^i$ are infinitesimal rigid motions
and $\widehat{u}\xy\colon \Omega\xy \to \R$, that 
$( \widehat{a}^{i_1}_j-\widehat{a}^{i_2}_j)\xy $ is constant in $\Omega\xy$ and uniformly bounded  \BBB w.r.t.\ \EEE $j$. 
This is in contradiction with \eqref{2403201949}.
Therefore, \eqref{2502201307} is proven. 

Moreover, for every $l$ there exists a unique $i \in \N$ such that 
\begin{equation}\label{2502201659}
(\widehat{u}_j - \widehat{a}_j^{i})\xy \to \widehat{u}\xy \quad\text{in }W^{1,p}_{\mathrm{loc}}(t_l, t_{l+1}) 
\end{equation}
and in particular the above convergence is locally uniform in $(t_l, t_{l+1})$.
Since $a_j^i$ are rigid motions, we also have that 
\[
\|\dot{u}\xy\|_{L^p(K)} \leq \liminf_{j\to \infty} \| (\dot{u}_h)\xy\|_{L^p(t_l, t_{l+1})} \text{ for every compact set }K \subset (t_l, t_{l+1})\,,
\]
so 
\[
\widehat{u}\xy \in SBV^p(\Omega\xy) \text{ and } J_{\widehat{u}\xy} \subset \{t_1,\dots, t_{M_y}\}\,.
\] 
This implies, with \eqref{0101182346},  that 
\begin{equation}\label{0201181710}
\ho\big(J_{\widehat{u}\xy} \cup \partial \mathcal{P}\xy \big)= \ho\big(J_{\widehat{u}\xy} \cap  (\mathcal{P}\xy)^{(1)} \big) + \ho(\partial \mathcal{P}\xy) \leq \liminf_{m\to \infty} \Big( \ho\big(J_{(\widehat{u}_m)\xy}\big) + \varepsilon \, \mathrm{I}\xy(u_m)\Big)\,.
\end{equation}
Notice that we have expressed $J_{\widehat{u}\xy}  \cup \partial \mathcal{P}\xy$ as the disjoint union $\big(J_{\widehat{u}\xy} \cap  (\mathcal{P}\xy)^{(1)}\big) \cup \partial \mathcal{P}\xy$, denoting
\[
 (\mathcal{P}\xy)^{(1)}:=\bigcup_{j \in \N} \big((P_j)\xy\big)^{(1)}
 \] (recall  \eqref{1305201259} \EEE and that $E^{(1)}$ denotes the point where $E\subset \mathbb{R}^d$ has density 1 \BBB w.r.t.\ \EEE $\mathcal{L}^d$;  above $d=1$\EEE).

Integrating over $y\in\Pi^\xi$ and using Fatou's lemma with \eqref{0101182324} we get
\begin{equation}\label{0201181712}
\begin{split}
\int \limits_{\Pi^\xi} &\Big[\ho\big(J_{\widehat{u}\xy} \cap  (\mathcal{P}\xy)^{(1)} \big) + \ho(\partial \mathcal{P}\xy)\Big] \,\mathrm{d}\hn(y) \\&\leq \liminf_{k\to\infty} \int\limits_{\Pi^\xi}\Big[ \ho\big(J_{(\widehat{u}_k)\xy}\big) + \varepsilon \, \mathrm{I}\xy(u_k) \Big]\,\mathrm{d}\hn(y)
\end{split}
\end{equation}
for $\hn$-a.e.\ $\xi \in \Sn$.
 In particular we deduce that each $P_j$ has finite perimeter (\textit{cf.}~\cite[Remark~3.104]{AFP}) and $ \sum_{j\in \N} \hn(\partial^* P_j) < +\infty$. \BBB This confirms that $\mathcal{P}$ is a Caccioppoli partition, as claimed at the end of Step~1. \EEE

Integrating further \eqref{0201181712} over $\xi \in \Sn$, by \eqref{0101182248}, \eqref{0101182137}, and \eqref{0101182244} we get
\begin{equation}\label{0201181718}
\hn\Big(J_u \cup \bigcup_{j\in \N} \partial^* P_j \BBB \cap \Omega \EEE \Big) \leq C\,M \varepsilon + \liminf_{h\to \infty} \hn(J_{u_h})\,,
\end{equation}
for a universal constant $C$. By the arbitrariness of $\varepsilon$, 
\eqref{eq:sciSalto} follows. 
\medskip
\paragraph*{\textbf{Step~3: Proof of \eqref{eq:convGradSym}}} In order to prove \eqref{eq:convGradSym} it is enough to combine what we have proven so far with the compactness and lower semicontinuity result \cite[Theorem~11.3]{DM13} (or \cite[Theorem~1.1]{CC18}).
In fact, the sequence $(u_h-a_h)_h$ is bounded in $GSBD^p(\Omega)$: by definition of $a_h^j$ we have that
\begin{equation*}
e(u_h-a_h)=e(u_h) \,,\quad J_{u_h-a_h}\subset J_{u_h} \cup \bigcup_{j \in \N} \partial^* P_j\,,
\end{equation*}
and 
we know that
 $(u_h)_h$ is bounded in $GSBD^p(\Omega)$ and $\sum_j \hn(\partial^* P_j)<+\infty$, by \eqref{eq:sciSalto}.
Since we know that $u_h-a_h \to u$, we are allowed to apply \cite[Theorem~11.3]{DM13} (or \cite[Theorem~1.1]{CC18}, knowing that the exceptional set $A$ therein is empty). We deduce 
\begin{equation*}
e(u_h)=e(u_h-a_h) \weak e(u) \quad\text{in }L^p(\Omega;\Mnn)\,,
\end{equation*}
so \eqref{eq:convGradSym} is proven and the general proof is concluded.
\end{proof}

\begin{remark}\label{rem:0403201804}
With the notation of Theorem~\ref{thm:main}, the sequence $(u_h-a_h)_h$ is bounded in $GSBD^p(\Omega)$. This is proven in Step~3.
\end{remark}
\begin{remark}\label{rem:3103201937}
We cannot directly apply \cite[Theorem~1.1]{CC18} to $u_h-a_h^j$ for every $j$ in Step~2. Indeed, we would obtain $\hn\big((J_u \cap P_j^{(1)}) \cup \partial^* P_j \BBB \cap \Omega\EEE) \leq \liminf_{h \to \infty} \hn(J_{u_h})$, but the $j$'s are countable many and we cannot localize on right-hand side, since the $P_j$'s are not open sets.
\end{remark}

%

\section{Lower semicontinuity and minimisation}\label{Sec3}


In this section we first prove our main lower semicontinuity result, concerning a class of free discontinuity functionals with general bulk and surface energy densities. In the second part we apply Theorem~\ref{thm:minimization} to the minimisation of free discontinuity functionals with general bulk and surface energy densities.


\begin{proof}[Proof of Theorem~\ref{thm:minimization}]
Up to a subsequence, we may assume that 
\begin{equation}\label{0603201834}
\liminf_{h\to \infty} E(u_h) = \lim_{h \to \infty} E(u_h) < +\infty\,.
\end{equation}
In view of the growth assumptions on $f$ and $g$, we have that $(u_h)_h$ is bounded in $GSBD^p(\Omega)$. Thus we may apply Theorem~\ref{thm:main} to find a subsequence (not relabelled), a Caccioppoli partition $\mathcal{P}$ of $\Omega$, a sequence of piecewise infinitesimal rigid motions $(a_h)_h$, and $u \in GSBD^p(\Omega)$ satisfying \eqref{2302202219} and \eqref{eqs:0203200917}.

By \eqref{0603201834} we obtain that, up to a further subsequence,  
\[
f(x,e(u_h)) \Ln\mres \Omega + g( x \EEE, [u_h], \nu_{u_h}) \hn \mres J_{u_h} =:\mu_h \wstar \mu \quad\text{in } \mathcal{M}_b^+(\Omega)
\]
as $h\to \infty$.
Therefore, by the Besicovitch derivation theorem and the Radon-Nikodym decomposition for $\mu$ (cf.\ \cite[Theorem~2.2]{AFP}), the result will follow from the estimates
\begin{equation}\label{0203201134}
\frac{\mathrm{d} \mu}{\mathrm{d} \Ln}(x_0) \geq f\big(x_0, e(u)(x_0)\big) \quad\text{for $\Ln$- a.e.\ } x_0 \in \Omega
\end{equation}
and
\begin{equation}\label{0603201227}
\begin{split}
\frac{\mathrm{d} \mu}{\dh}(x_0) &\geq g\big( x_0, \EEE [u](x_0), \nu_u(x_0) \big) \quad\text{for }\hn\text{-a.e.\ }x_0 \in J_u \cap \mathcal{P}^{(1)}\,,\\
\frac{\mathrm{d} \mu}{\dh}(x_0) &\geq g_\infty( x_0, \EEE \nu_{\mathcal{P}}) \quad\text{for }\hn\text{-a.e.\ }x_0 \in \partial^*\mathcal{P} \BBB \cap \Omega\EEE\,.
\end{split}
\end{equation}
\medskip
\paragraph*{\textbf{Step~1: Proof of \eqref{0203201134}.}}
We divide this step into different substeps.
\medskip
\paragraph*{\textit{Substep~1.1: Choice of the blow up point $x_0$ and first properties.}}
 We pick $x_0$ in a subset of $\Omega$ of full $\Ln$-measure, satisfying the following four criteria. 
First, we notice that by the definition of Radon-Nikodym derivative and \cite[Theorem~2.2]{AFP} we have that for $\Ln$-a.e.\ $x_0 \in \Omega$ 
\begin{equation}\label{2504201018}
\frac{\mathrm{d} \mu}{\mathrm{d} \Ln}(x_0)= \lim_{\varrho\to 0^+}\frac{\mu(B_\varrho(x_0))}{\gamma_n\, \varrho^n}\,.
\end{equation} 
 Second, in \cite{CagChaSca20} it is proven that every function in $GSBD^p$ is approximately differentiable $\Ln$-a.e., namely that  for $\Ln$-a.e.\ $x_0 \in \Omega$ there exists $\nabla u(x_0) \in \mathbb{M}^{n\times n}$ (such that $e(u)(x_0)= (\nabla u(x_0))^{\mathrm{sym}}$ for a.e. $x_0$) for which it holds 
\begin{equation}\label{0203201223}
\aplim_{x\to x_0} \frac{  |u(x)-u(x_0)- \nabla u(x_0) (x-x_0)|}{|x-x_0|} =0\,.
\end{equation} 
Third, we take 
\begin{equation}\label{0403201320}
x_0 \in \mathcal{P}^{(1)}\,,
\end{equation} 
which is a set of full $\Ln$-measure in $\Omega$.
The fourth and last criterion employed in the choice of $x_0$ is based on the properties of $f$.
Since $f$ is a Carathéodory function,  arguing as in \cite[proof of Theorem~1.2]{Ebo05}, \EEE by  Scorza Dragoni  Theorem (see, e.g., \cite{EkeTem76}, p.\ 235) one deduces that there exists $F \subset \Omega$ with $\Ln(\Omega\sm F)=0$ such that for any $x_0 \in F$ there exists a compact set $K_{x_0}\subset \Omega$
(depending on $x_0$) such that 
\begin{equation}\label{0603201829}
f|_{K_{x_0} {\times} \Mnn} \quad\text{is continuous in }K_{x_0} {\times} \Mnn \quad\text{and} \quad {x_0} \in K_{x_0} \cap K_{x_0}^{(1)}\,.
\end{equation}
 Then the set of points $x_0$ satisfying \eqref{2504201018}, \eqref{0203201223}, \eqref{0403201320}, and \eqref{0603201829} is of full $\Ln$-measure in $\Omega$. Let us choose $x_0$ in this set. 

Let us fix a sequence $(\varrho_k)_k$ converging to 0 such that $\mu(\partial B_{\varrho_k}(x_0))=0$ for every $k$ (in fact this is true for any $\varrho>0$ except at most countable many). Then,  by \eqref{2504201018}    we have that 
\begin{equation}\label{0203201239}
\begin{split}
\gamma_n\,\frac{\mathrm{d} \mu}{\mathrm{d} \Ln}(x_0)&= \lim_{k\to \infty}\lim_{h\to \infty}\frac{\mu_h(B_{\varrho_k}(x_0))}{\varrho_k^n}\\
&= \lim_{k\to \infty}\lim_{h\to \infty}\frac{1}{\varrho_k^n}\Bigg\{\int_{B_{\varrho_k}(x_0)} f(x, e(u_h)(x)) \dx + \int_{J_{u_h} \cap B_{\varrho_k}(x_0)} g( x, \EEE[u_h], \nu_{u_h}) \dh  \Bigg\} \,.
\end{split}
\end{equation} 
\BBB Moreover, \eqref{2202201910} gives that $y\mapsto(u_h-a_h)(x_0+\varrho_k y)$ converge pointwise in $B_1$ to $y\mapsto u(x_0+\varrho_k y)$, and
by \eqref{0203201223}, \eqref{0403201320} it holds that $y\mapsto  \frac{u(x_0+\varrho_k y)-u(x_0)}{\varrho_k} $ converges to $y\mapsto \nabla u(x_0)\,y$ pointwise in $B_1$. 
Therefore
\begin{equation}\label{0203201224}
\lim_{k\to \infty}\lim_{h\to \infty}u_{k,h}^{x_0}\to \nabla u(x_0)\,\cdot \quad\text{in }L^0(B_1;\Rn), \quad u_{k,h}^{x_0}(y):=\frac{(u_h-a_h)(x_0+\varrho_k y)-u(x_0)}{\varrho_k}\,.
\end{equation}
\EEE
Furthermore, $\lim_{k\to \infty} \varrho_k^{-(n-1)}\hn(\partial^*\mathcal{P}\cap B_{\varrho_k}(x_0))=0$, and so 
\begin{equation}\label{0203201227}
 \lim_{k\to \infty}\lim_{h\to \infty}\frac{\hn(J_{a_h} \cap B_{\varrho_k}(x_0))}{\varrho_k^{n-1}}=0\,.
\end{equation}
\medskip

\paragraph*{\textit{Substep~1.2: Blow up argument: change of variables.}} We perform a blow up procedure in correspondence of a point $x_0 \in \Omega$ chosen as above, in order to prove \eqref{0203201134}.

Let us consider the functions $u_{k,h}^{x_0}$, defined in \eqref{0203201224}.
We notice that \eqref{0203201239} and \BBB ($g_1$) \EEE
imply that for a suitable $\widetilde{C}>0$
\begin{equation*}
\limsup_{k\to \infty}\limsup_{h\to \infty}\frac{1}{\varrho_k^n} \hn(J_{u_h} \cap B_{\varrho_k}(x_0)) \leq \widetilde{C}\,.
\end{equation*}
Together with \eqref{0203201227}, by a change of variable this gives that
\begin{equation}\label{0603202037}
\limsup_{k\to \infty} \limsup_{h\to \infty} \hn(J_{u_{k,h}^{x_0}})=0\,.
\end{equation}
Setting
\begin{equation*}
f_k(y,\xi):= f(x_0+\BBB \varrho_k \EEE y, \xi)\,,
\end{equation*} 
by \eqref{0203201239} we obtain that
\begin{equation}\label{0603202028}
\gamma_n\,\frac{\mathrm{d} \mu}{\mathrm{d} \Ln}(x_0)\geq \limsup_{k\to \infty}\limsup_{h\to \infty} \int_{B_1} f_k(y, e(u_{k,h}^{x_0})(y)) \, \mathrm{d}y \,.
\end{equation} 
Then a diagonal argument allows to define functions $v_k:= u_{k,h_k}^{x_0}$ such that \eqref{0203201224}, \eqref{0603202037}, and \eqref{0603202028} hold true for $v_k$ (and considering the limit or the $\limsup$ only in $k$).


Eventually we observe that, due to \eqref{0603201829}, 
\begin{equation}\label{0603202050}
\lim_{k\to \infty} f_k(y, \xi)=f(x_0,\xi) \quad \text{for a.e.\ }y\in B_1,\, \text{ locally uniformly in }\Mnn\,.
\end{equation}
In fact, the pointwise convergence follows from the fact that $x_0 \in K_{x_0} \cap K_{x_0}^{(1)}$, and the local uniform convergence in $\Mnn$ by the continuity of $f$ in $K_{x_0} \times \Mnn$.
\medskip
\paragraph*{\textit{Substep~1.3: Blow up argument: lower semicontinuity.}}
Let us fix $\ol \sigma \in (0,1)$.
In correspondence of $\ol \sigma$ we find positive constants $\eta(\ol\sigma)$ and  $C$ \EEE
such that the conclusion of  Theorem~\ref{th:KornGSBD} \EEE
holds. Fix also $\delta\in (0, \frac{\eta(\ol\sigma)}{\widetilde{C}})$, where $\widetilde{C}$ is the constant \BBB from the equation above \EEE \eqref{0603202037}.
We notice that, up to consider a subsequence (not relabelled), we may assume that 
\begin{equation}\label{1003201107}
\sum_{k \in \N} \hn(J_{v_k}) <  \big(  C^{-1} \EEE \wedge 1 \big) \,\delta\,.
\end{equation}
In particular, we have that $\hn(J_{v_k} \cap B_1) < \eta(\ol \sigma)$ for every $k$,
and we may apply 
\BBB Theorem~\ref{th:KornGSBD} \EEE to the functions $v_{k} \in GSBD^p(B_1)$. 
This provides functions $w_k \in GSBD^p(B_1) \cap W^{1,p}(B_{\ol\sigma};\Rn)$ and sets of finite perimeter $\omega_k \subset B_1$ such that 
\begin{equation}\label{1003201110}
w_k=v_k \text{ in } B_1 \sm \omega_k,\quad \hn(\partial^* \omega_k) < C\, \EEE \hn(J_{v_k}), 
\end{equation}
 $\hn(J_{w_k})\leq \hn(J_{v_k})$, and
\begin{equation}\label{0903200002}
\int_{B_1} |e(w_k)|^p \dx \leq  2 \EEE \int_{B_1} |e(v_k)|^p \dx\,. 
\end{equation}
By \eqref{1003201107}, \eqref{1003201110}, and the Isoperimetric Inequality
$\Ln(\omega_k) \leq \BBB \hn(\partial^* \omega_k)^{\frac{n}{n-1}}\EEE $, 
we have that
\begin{equation}\label{1003201118}
\Ln(\omega_\delta) < \delta  \qquad \text{ for } \ \omega_\delta:= \bigcup_{k\in \N}\omega_k\,, 
\end{equation} 
and
\begin{equation}\label{1003201119}
w_k=v_k \text{ in } B_1 \sm \omega_\delta \quad\text{for every }k\in \N \,.
\end{equation}
\BBB We have that $w_k$ are equibounded in $W^{1,p}(B_{\ol\sigma};\Rn)$: $w_k \in W^{1,p}(B_{\ol\sigma};\Rn)$, $\|\nabla w_k\|_{L^p(B_{\ol\sigma})}\leq C \|e(v_k)\|_{L^p(B_1)}$ by \eqref{0903200002} and Korn's Inequality, and $w_k$ pointwise converge by \eqref{0203201224}. \EEE
\BBB Moreover, by \eqref{1003201110}  we deduce that $\Ln(\{w_k \neq v_k\})\to 0$, so that \EEE 
\begin{equation*}
w_k(y) \weak \nabla u (x_0) \,y \quad\text{in }W^{1,p}(B_{\ol\sigma};\Rn)\,.
\end{equation*}

We now perform a further approximation, through a sequence of equi-Lipschitz functions. This is done for two reasons: first, to employ \eqref{0603202050} since therein the convergence holds for $\xi$ in compact sets; second, to pass to the limit in the integral of $f(x_0, e(w_k))$ over the set $B_{\ol\sigma} \sm \omega_\delta$, which is not in general open and so the semicontinuity theorem in \cite{AceFus84} does not apply directly.
Then we recall, adapting to the present case, what proven in \cite[Proposition~3.1]{Ebo05}, in the spirit of \cite{AceFus84} and \cite{Amb94}.

In correspondence to $\delta$, there exist a  Borel set $E_\delta$ with $\Ln(E_\delta)< \delta$ (this replaces the sequence $(E_k)_k$  with $\Ln(E_k)\to 0$ as $k\to \infty$ in \cite[Proposition~3.1]{Ebo05}) and for every $m$ and $k\in \N$ there exist $\widehat{w}_{k,m} \in W^{1,\infty}(B_{\ol\sigma};\Rn)$, $E_{k,m} \subset B_{\ol\sigma}$ Borel sets such that
\begin{equation}\label{0803202300}
\|\widehat{w}_{k,m}\|_{L^\infty} + \mathrm{Lip}(\widehat{w}_{k,m}) \leq C(n, B_1)m,\qquad
 \widehat{w}_{k,m}=w_k \text{ in }B_{\ol\sigma}\sm E_{k,m}\,,
 \end{equation}
and, up to extracting a subsequence \BBB w.r.t.\ \EEE $k$, $\widehat{w}_{k,m} \wstar \widehat{w}_m$ in $W^{1,\infty}(B_{\ol\sigma};\Rn)$ for every $m$, with
\begin{equation}\label{0803201858}
\begin{split}
& \lim_{m\to \infty} \limsup_{k\to \infty} \int_{E_{k,m}\sm E_\delta} \Big(1+ f_k(y, e(\widehat{w}_{k,m})) \Big) \mathrm{d}y =0\,,\\ 
 &\lim_{m\to \infty} m^p \,\Ln(A_m)=0 \,, \quad\text{for }A_m:=\{ \widehat{w}_m \neq  \nabla u(x_0) \,\cdot \} \cap (B_{\ol\sigma}\sm E_\delta)\,.
 \end{split}
\end{equation}

It follows that
\begin{equation}\label{1003201150}
\begin{split}
\int_{B_1} f_k(y, e(v_k)) \, \mathrm{d}y &\geq \int_{B_{\ol\sigma}\sm \omega_\delta} f_k(y, e(w_k)) \, \mathrm{d}y \geq  \int_{B_{\ol\sigma}\sm (E_\delta\cup \omega_\delta \cup E_{k,m})} f_k(y, e(\widehat{w}_{k,m})) \, \mathrm{d}y
\\& = \int_{B_{\ol\sigma}\sm (E_\delta\cup \omega_\delta)} f_k(y, e(\widehat{w}_{k,m})) \, \mathrm{d}y  - \int_{E_{k,m} \sm (E_\delta\cup \omega_\delta)} f_k(y, e(\widehat{w}_{k,m})) \, \mathrm{d}y
\end{split}
\end{equation}
recalling that $f_k \geq 0$, $w_k=v_k$ in $B_1 \sm \omega_\delta$, and $\widehat{w}_{k,m}=w_k$ in $B_{\ol\sigma}\sm E_{k,m}$.
We now use the fact that $\widehat{w}_{k,m} \wstar \widehat{w}_m$ in $W^{1,\infty}(B_{\ol\sigma};\Rn)$ (so that $(\widehat{w}_{k,m})_k$ is a sequence of equi-Lipschitz functions and $\big(f(x_0, e(\widehat{w}_{k,m}))\big)_k$ is equi-integrable in $B_{\ol\sigma}$) and  \eqref{0603202050}, to deduce that
\begin{equation}\label{0803201838}
\begin{split}
\liminf_{k\to \infty}  \int_{B_{\ol\sigma}\sm (E_\delta\cup \omega_\delta)} f_k(y, e(\widehat{w}_{k,m})) \, \mathrm{d}y & = \liminf_{k\to \infty}  \int_{B_{\ol\sigma}\sm (E_\delta\cup \omega_\delta)} f(x_0, e(\widehat{w}_{k,m})) \, \mathrm{d}y\,,
\\&  \geq \int_{B_{\ol\sigma}\sm (E_\delta\cup \omega_\delta)} f(x_0, e(\widehat{w}_m)) \, \mathrm{d}y\,.
 \end{split}
\end{equation}
We observe that in the equality above we used \eqref{0603202050}, and to prove the latter estimate it is enough to apply Morrey's Lower Semicontinuity Theorem in an arbitrary open set containing $B_{\ol \sigma}\sm (E_\delta\cup \omega_\delta)$ and observe that for any $\varepsilon>0$, thanks to the equi-integrability,
we can find an open set $B'_\varepsilon \supset  B_{\ol \sigma}\sm (E_\delta\cup \omega_\delta)$ such that the integrals of $f(x_0, e(\widehat{w}_{k,m}))$ (for every $k$) and $f(x_0, e(\widehat{w}_m))$ over $\BBB B'_\varepsilon \EEE \sm \big( B_{\ol \sigma}\sm (E_\delta\cup \omega_\delta) \big)$ are less than $\varepsilon$.

The second estimate in \eqref{0803201858}, \eqref{0803202300}, and $f\geq 0$ imply that
\begin{equation}\label{0803202304}
\int_{B_{\ol\sigma}\sm (E_\delta\cup \omega_\delta)} f(x_0, e(\widehat{w}_m)) \, \mathrm{d}y \geq \Ln\big(B_{\ol\sigma}\sm (E_\delta\cup \omega_\delta \cup A_m) \big) f(x_0, e(u)(x_0)) \,.
\end{equation}

Moreover, employing again $f_k \geq 0$,
\begin{equation}\label{1003201504}
\int_{E_{k,m} \sm (\omega_\delta \cup E_\delta)} f_k(y, e(\widehat{w}_{k,m})) \, \mathrm{d}y \leq \int_{E_{k,m} \sm E_\delta} f_k(y, e(\widehat{w}_{k,m})) \, \mathrm{d}y \,.
\end{equation}

Collecting \eqref{1003201150}, \eqref{0803201838}, \eqref{1003201504}, \eqref{0803202304}, \eqref{0803201858}, and passing to the $\liminf$ in $k$ and to the limit in $m$, we obtain that
\begin{equation*}
\liminf_{k\to \infty} \int_{B_1} f_k(y, e(v_k)) \, \mathrm{d}y \geq \Ln\big(B_{\ol\sigma}\sm (E_\delta\cup \omega_\delta) \big) f(x_0, e(u)(x_0))> \BBB (\gamma_n\ol\sigma^n -2\delta)\EEE f(x_0, e(u)(x_0))\,.
\end{equation*}
Passing to the limit first as $\delta\to 0$ and then as $\ol \sigma \to 1$, by \eqref{0603202028} (recall the definition $v_k=u_{k,h_k}^{x_0}$) we deduce \eqref{0203201134}.

\medskip
\paragraph*{\textbf{Step~2: Proof of \eqref{0603201227}.}} We denote  $J'_u:= (J_u \cap \mathcal{P}^{(1)}) \cup (\partial^* \mathcal{P}\BBB\cap \Omega)\EEE$.
\medskip
\paragraph*{\textit{Substep~2.1: Choice of the blow up point $x_0$ and first properties.}}
Since $J'_u$ is countably rectifiable and thanks to \eqref{0106172148},
 for $\hn$-a.e.\ $x_0 \in J'_u$ there exist $u^+(x_0)$, $u^-(x_0)\in \Rn$, $\nu_0 \in \Sn$ such that
\begin{equation}\label{1003202359}
\aplim_{\substack{x \in (Q_\varrho^{\nu_0}(x_0))^\pm\\ x\to x_0}} u(x)=u^{\pm}(x_0)\,.
\end{equation}  
Notice that 
$\nu_0$ denotes $\nu_u(x_0)$ if $x_0 \in J_u$ and the outer normal to $P_i$ at $x_0$, if $x_0 \in \partial^* \mathcal{P}\BBB \cap \Omega\EEE$ and $x_0 \in \partial^*P_i \cap \partial^*P_j$ for $i<j$.
 We remark that $u^+(x_0)\neq u^-(x_0)$ for 
 $x_0 \in J_u$. 

 Moreover, since $\mu$ is a positive bounded Radon measure and $J'_u$ is countably rectifiable, there exists the Radon-Nikodym derivative of $\mu$ \BBB w.r.t.\ \EEE $\hn\mres J'_u$ (which is $\sigma$-finite) and it holds (see e.g.\ \cite[Theorems~1.28 and 2.83]{AFP})
\begin{equation}\label{0203201213}
\frac{\mathrm{d} \mu}{\mathrm{d} \hn}(x_0)= \lim_{\varrho\to 0^+}\frac{\mu(Q^{\nu_0}_\varrho(x_0))}{\varrho^{n-1}}\quad\text{for }\hn\text{-a.e.\ } x_0 \in J'_u\,.
\end{equation}


We thus fix $x_0$ such that both \eqref{1003202359} and \eqref{0203201213}  hold for $x_0$. Moreover, let us consider the sets $D\subset \Sn$ and $N \subset \partial^* \mathcal{P}$ given by Lemma~\ref{le:2503201833} applied to $F=\partial^* \mathcal{P}\BBB \cap \Omega\EEE$, and fix $x_0 \notin N$.

Recalling the pointwise convergence of $u_h-a_h$ to $u$, by
a change of variables we obtain from \eqref{1003202359} that
\begin{equation}\label{1003202244}
 (u_h-a_h)(x_0+\varrho_k \cdot ) \to u_0:=u^+(x_0) \chi_{Q_1^{\nu_0,+}} + u^-(x_0) \chi_{Q_1^{\nu_0,-}} \quad\text{in }L^0(Q_1^{\nu_0};\Rn)
\end{equation}
as $h\to \infty$, $k\to \infty$. \BBB (Recall the notation for half cubes in Section~\ref{Sec1}.) \EEE
Analogously to \eqref{0203201239}, we fix a vanishing sequence $(\varrho_k)_k$ with $\mu(\partial Q^{\nu_0}_{\varrho_k}(x_0))=0$, and then
\begin{equation}\label{1003202145}
\begin{split}
\frac{\mathrm{d} \mu}{\mathrm{d} \hn}(x_0)&= \lim_{k\to \infty}\lim_{h\to \infty}\frac{\mu_h(Q^{\nu_0}_{\varrho_k}(x_0))}{\varrho_k^{n-1}}\\
&= \lim_{k\to \infty}\lim_{h\to \infty}\frac{1}{\varrho_k^{n-1}}\Bigg\{\int_{Q^{\nu_0}_{\varrho_k}(x_0)} f(x, e(u_h)(x)) \dx + \int_{J_{u_h} \cap Q^{\nu_0}_{\varrho_k}(x_0)} g( x, \EEE [u_h], \nu_{u_h}) \dh  \Bigg\}\,.
\end{split}
\end{equation}
\medskip
\paragraph*{\textit{Substep~2.2: Blow up argument for $x_0 \in J_u \cap \mathcal{P}^{(1)}$.}} Given $x_0 \in J_u \cap \mathcal{P}^{(1)}$, there exists $j\in \N$ such that $x_0 \in \BBB P_j^{(1)}$. \EEE Then
\begin{equation}\label{2103201324}
\lim_{h\to +\infty} |a_h - a_h^j| = 0 \quad\Ln\text{-a.e.\ in }Q_1^{\nu_0}\,,\quad 
\end{equation}
for $a_h^j$ the infinitesimal rigid motion corresponding to \BBB $P_j$, \EEE cf.\ \eqref{2302202219}. 
Up to choosing a subsequence $h_k$ in \eqref{1003202244}, by \eqref{2103201324} we get  
\begin{equation}\label{2003201714}
\widetilde{v}_k:=(u_{h_k}-a_{h_k}^j)(x_0+\varrho_k \cdot) \to u_0:=u^+(x_0) \chi_{Q_1^{\nu_0,+}} + u^-(x_0) \chi_{Q_1^{\nu_0,-}} \quad\text{in }L^0(Q_1^{\nu_0};\Rn)\,.
\end{equation}

By \eqref{1003202145}  and assumptions ($g_1$), ($g_3$) \EEE we obtain that
\begin{equation}\label{1103200009}
\begin{split}
\frac{\mathrm{d} \mu}{\mathrm{d} \hn}(x_0)&=  \lim_{k\to \infty} \Bigg\{\int_{Q_1^{\nu_0}} f(x, e(\widetilde{v}_k)(x)) \dx + \int_{J_{\widetilde{v}_k} \cap Q_1^{\nu_0}} g( x_0, \EEE [\widetilde{v}_k], \nu_{\widetilde{v}_k}) \dh \Bigg\} \,.
\end{split}
\end{equation}
We remark that above we used that $g$ does not depend separately on the two traces $v^+$ and $v^-$ but only on $[v]$. This allowed us to infer that for any function $v$ and infinitesimal rigid motion $a$ the surface part evaluated on $v$ is equal to the surface part evaluated on $v-a$.

By the growth assumptions on $f$ and $g$ it follows that $(\widetilde{v}_k)_k$ converges weakly in $GSBD^p(Q_1^{\nu_0})$ to $u_0$.

Therefore, ($g_4$) and \eqref{1103200009} imply that
\begin{equation*}
\begin{split}
g\big( x_0, \EEE [u](x_0), \nu_u(x_0)  \big)  &= \int_{J_{u_0}} g_{ x_0 \EEE}([u_0], \nu_0) \dh \\&\leq \liminf_{k\to \infty} \int_{J_{\widetilde{v}_k}\cap Q_1^{\nu_0}} g_{ x_0 \EEE}([\widetilde{v}_k], \nu_{\widetilde{v}_k}) \dh \leq \frac{\mathrm{d} \mu}{\mathrm{d} \hn}(x_0)\,.
\end{split}
\end{equation*}
\medskip
\paragraph*{\textit{Substep~2.3: Blow up argument for $x_0 \in \partial^* \mathcal{P}\BBB \cap\Omega\EEE$.}} Assume that $x_0 \in \partial^* P_i \cap \partial^* P_j$, for $i<j$.
Let us take a subsequence $h_k$ such that the convergences in \eqref{1003202244} and \eqref{1003202145} hold along $h_k$, as $k\to \infty$. Then we denote
\begin{equation*}
v_k(y):=u_{h_k}(x_0+\varrho_k y)\quad\text{for }y \in Q_1^{\nu_0}\,.
\end{equation*}
These functions satisfy (by a change of variables in \eqref{1003202145})
\begin{equation}\label{2303201131}
\frac{\mathrm{d} \mu}{\mathrm{d} \hn}(x_0)= \lim_{k\to \infty}\Bigg\{\int_{Q_1^{\nu_0}} f(x, e(v_k)(x)) \dx + \int_{J_{w_k} \cap Q_1^{\nu_0}} g( x, \EEE [v_k], \nu_{v_k}) \dh  \Bigg\}
\end{equation}
and, by \eqref{1003202244},
\begin{equation*}\label{2303201132}
 (u_{h_k}-a_{h_k})(x_0+\varrho_k \cdot) \to u_0 \quad\text{in }L^0(Q_1^{\nu_0};\Rn)
\end{equation*}
as $k\to \infty$. 
In particular, since $x_0 \in \partial^* P_i \cap \partial^* P_j$, setting $a^+_k(y):=a^j_{h_k}(x_0+\varrho_k y)$ and $a^-_k(y):=a^i_{h_k}(x_0+\varrho_k y)$ for $y \in Q_1^{\nu_0}$, we have that
\begin{equation}\label{2403201239}
\begin{split}
v_k - a^+_k=(u_{h_k}-a^j_{h_k})(x_0+\varrho_k \cdot) &\to u^+(x_0) \quad\text{in }L^0(Q_1^{\nu_0,+};\Rn)\,,
\\
v_k - a^-_k= (u_{h_k}-a^i_{h_k})(x_0+\varrho_k \cdot) &\to u^-(x_0) \quad\text{in }L^0(Q_1^{\nu_0,-};\Rn)\,.
\end{split}
\end{equation}
Moreover, in view of the choice $x_0 \notin N$ 
it holds that
$|(a_k^+-a_k^-)\cdot \xi|\to +\infty$ uniformly in $Q_1^{\nu_0}$, as $k\to \infty$, for any $\xi \in D$ (recall that $N$ and $D$ are given by Lemma~\ref{le:2503201833}).  Since 
$e(a_k^+-a_k^-)=0$, we get 
\begin{equation}\label{2603201108}
|(\widehat{a}_k^+-\widehat{a}_k^-)\xy|\equiv|(\widehat{a}_k^+-\widehat{a}_k^-)\xy(0)| \to +\infty \quad\text{ as $k\to \infty$, for any $\xi \in D$}\,.
\end{equation}
\medskip
\paragraph*{\textbf{Case $g_\infty( x_0, \EEE \nu_0) \in \R$.}}
Assume that $g_\infty( x_0, \EEE \nu_0) \in \R$,  so that $g_\infty$ takes finite values, \EEE and fix $\eta>0$ small. We find $\xi_0=\xi_0(\nu_0,\eta) \in D \subset \Sn$ such that $\xi_0$ satisfies \eqref{2502200925} and 
\begin{equation}\label{2303202056}
\bigg|g_\infty( x_0, \EEE \nu_0)-\frac{|\nu_0\cdot \xi_0|}{g^*_{ x_0, \EEE \infty}(\xi_0)}\bigg|<\eta\,,
\end{equation} \EEE
where $g^*_{ x_0, \EEE \infty}$ is the dual norm of $g_\infty( x_0, \EEE \cdot)$,  given by $\phi^*(\xi):=\sup_{\phi(\nu)\leq 1} |\nu \cdot \xi|$. \EEE 
This is done by choosing a vector $\ol\xi$ in $\Sn$ such that $g_\infty( x_0, \EEE \nu_0)=\frac{|\nu_0\cdot \ol\xi|}{g^*_{ x_0, \EEE \infty}(\ol\xi)}$
and by continuity, using that $D$ is dense in $\Sn$.


By ($g_5$) there is a function $\kappa\colon [0,+\infty) \to [0,+\infty)$ with $\lim_{t\to +\infty}\kappa(t)=0$ such that
\begin{equation*}
g( x, \EEE y, \nu) > g_\infty( x \EEE, \nu)-\kappa(t)\quad \text{for every $x\in \Omega$, $|y|>t$, and $\nu\in \Sn$,}
\end{equation*}
 and, from ($g_3$),
 $g( x, \EEE y, \nu) > g_\infty( x_0 \EEE, \nu)-\kappa(t)$ in a neighbourhood of $x_0$. \EEE
By the definition of dual norm,
($g_5$), \eqref{2303202056}, and since  $|[v_k]\cdot \xi_0| \leq |[v_k]|$, \EEE
we get
\begin{equation}\label{2303201142}
\begin{split}
g( x, \EEE [v_k], \nu_{v_k}) &\geq (g_\infty( x_0, \EEE \nu_{v_k} )-\kappa(t))
  \chi_{ \{ |[v_k]|>t\} } 
\geq \Bigg( \frac{|\nu_{v_k} \cdot \xi_0|}{g^*_{ x_0, \EEE \infty}(\xi_0)} - \kappa(t) -\eta \Bigg) \chi_{ |\{ [v_k]\cdot \xi_0|>t\} }
\end{split}
\end{equation}
 $\hn$-a.e.\ in $J_{v_k}\cap Q_1^{\nu_0}$. 
We observe that 
\begin{equation}\label{2303201954}
\limsup_{k\to \infty} \hn(J_{v_k} \cap Q_1^{\nu_0})=:L< +\infty\,,
\end{equation} 
since $g$ takes values in $[c,+\infty)$. 
Then, using also \eqref{2303201131} and \eqref{2303201142}  we obtain that
\begin{equation}\label{2303201235}
\begin{split}
\frac{\mathrm{d} \mu}{\mathrm{d} \hn}(x_0)& \geq \liminf_{k\to \infty}\Bigg\{ \int_{Q_1^{\nu_0}} \frac{|e(v_k) \xi_0 \cdot \xi_0|^p}{C} \dx + \int_{J_{v_k} \cap Q_1^{\nu_0}} \Bigg( \frac{|\nu_{v_k} \cdot \xi_0|}{g^*_{ x_0, \EEE \infty}(\xi_0)}  \chi_{ \{ |[v_k]\cdot \xi_0| \EEE>t\}} + \varepsilon \Bigg)\dh \Bigg\} \\& \hspace{1em}
-(\kappa(t)+\eta +\varepsilon)L
\\& = \liminf_{k\to+\infty} \int_{\Pi^{\xi_0}}  F_{y,t}^{\xi_0,\varepsilon}((\widehat{v}_k)\xoy) \dh(y) -(\kappa(t)+\eta +\varepsilon)L
\end{split}
\end{equation}
with
\begin{equation*}
F_{y,t}^{\xi_0,\varepsilon}(v):= \frac{1}{C}\int_{(Q_1^{\nu_0})\xoy}|v'(s)|^p \,\mathrm{d}s \ + \mathcal{H}^0(\{ s\colon  |[v](s)| \EEE >t \}) \frac{1}{g^*_{ x_0, \EEE \infty}(\xi_0)} +\varepsilon \mathcal{H}^0(J_v) \quad
\end{equation*}
for $v\colon (Q_1^{\nu_0})\xoy \to \R$.
We observe that the second relation in \eqref{2303201235} follows from the Area Formula (cf.\ e.g.\ \cite[(12.4)]{Sim84}) and the slicing property \eqref{slicingSaltoGSBD}. 
 
Fatou's lemma and \eqref{2303201235} give that $\liminf_k  F_{y,t}^{\xi_0,\varepsilon}((\widehat{v}_k)\xoy)<+\infty$ for $\hn$-a.e.\ $y \in \Pi^{\xi_0}$, so we may find, for $\hn$-a.e.\ $y \in \Pi^{\xi_0}$, a subsequence $\widehat{v}_m=\widehat{v}_{k_m}$ (depending on $y$) such that
\begin{equation}\label{2603201226}
\lim_{m\to \infty} F_{y,t}^{\xi_0,\varepsilon}((\widehat{v}_m)\xoy)= \liminf_{k\to \infty} F_{y,t}^{\xi_0,\varepsilon}((\widehat{v}_k)\xoy)\,,\quad \mathcal{H}^0\big(J_{(\widehat{v}_m)\xoy}\big)= N_y \in \N\,.
\end{equation}
Recalling \eqref{2403201239},
we may also choose the subsequence $(k_m)_m$ such that, denoting by $(\widehat{v}_m- \widehat{a}^\pm_m)\xoy$ the functions $(\widehat{v}_{k_m}- \widehat{a}^\pm_{k_m})\xoy$, it holds
\begin{equation}\label{2603201215}
(\widehat{v}_m- \widehat{a}^\pm_m)\xoy \to u^\pm(x_0) \cdot \xi_0 \quad\text{in }L^0\big((Q_1^{\nu_0,\pm})\xoy\big)\,.
\end{equation}

We now claim that there exists $\ol m \in \N$ such that 
\begin{equation}\label{2603201211}
\Big\{ s \in (Q_1^{\nu_0,\pm})\xoy \colon  |[(\widehat{v}_m)\xoy](s)| \EEE >t \Big\} \neq \emptyset \quad\text{for }m\geq \ol m\,.
\end{equation}
Indeed, let us argue by contradiction assuming that \eqref{2603201211} is not true. Then, by \eqref{2603201226},
\begin{equation}\label{2903201858}
 \mathrm{D}\big((\widehat{v}_m)\xoy\big) \big( (Q_1^{\nu_0,\pm})\xoy\big) \EEE\leq \int_{(Q_1^{\nu_0})\xoy}\Big|\big((\widehat{v}_m)\xoy\big)'(s)\Big|\, \mathrm{d}s + t\, N_y \leq \widehat{C} \,,
\end{equation}
for a suitable $\widehat{C}>0$ independent of $m$.
Therefore, for any $s^+ \in (Q_1^{\nu_0,+})\xoy$, $s^- \in (Q_1^{\nu_0,-})\xoy$,
\begin{equation*}
\begin{split}
 | \EEE(\widehat{a}_m^+-\widehat{a}_m^-)\xoy | \EEE& \equiv\EEE  | \EEE(\widehat{a}_m^+)\xoy(s^+)-(\widehat{a}_m^-)\xoy(s^-) | \EEE
\\&
= \Big| \EEE(\widehat{v}_m- \widehat{a}^-_m)\xoy(s^-) - (\widehat{v}_m- \widehat{a}^+_m)\xoy(s^+) + \Big((\widehat{v}_m)\xoy(s^+) - (\widehat{v}_m)\xoy(s^-)\Big) \Big| \EEE\,.
\end{split}
\end{equation*}
From the identity above we obtain a contradiction for $m$ large enough, since the left-hand side tends to $+\infty$ by \eqref{2603201108}
 while the right-hand side is bounded by \eqref{2603201215} and \eqref{2903201858}.
This proves \eqref{2603201211}.

In particular, \eqref{2603201211} implies that $\lim_{m\to \infty} F_{y,t}^{\xi_0,\varepsilon}((\widehat{v}_m)\xoy)\geq \frac{1}{g^*_{ x_0, \EEE \infty}(\xi_0)}$. Recalling \eqref{2603201226} and integrating in $\Pi^{\xi_0}=\Pi^{\xi_0}(Q_1^{\nu_0})$, by Fatou's lemma, \eqref{2303201235}, and the arbitrariness of $t$, $\eta$, $\varepsilon$, we get
\begin{equation}\label{2803201204}
\frac{\mathrm{d} \mu}{\mathrm{d} \hn}(x_0) \geq \frac{ | \EEE\nu_0\cdot \xi_0 | \EEE}{g^*_{ x_0, \EEE \infty}(\xi_0)}\,.
\end{equation}
The second estimate in \eqref{0603201227} follows now by \eqref{2303202056} and the arbitrariness of $\eta$.

\medskip
\paragraph*{\textbf{Case $g_\infty \equiv \EEE+\infty$.}} We have to prove (with the usual notation $\nu_0=\nu(x_0)$) that  
\begin{equation}\label{2803201153}
\hn(\partial^*\mathcal{P}\BBB \cap\Omega\EEE)=0\,.
\end{equation}
Assume by contradiction that there is $(\partial^*\mathcal{P}\BBB \cap\Omega\EEE)\sm N \neq \emptyset$ (for $N$ given by Lemma~\ref{le:2503201833}), and let $x_0 \in (\partial^*\mathcal{P}\BBB \cap\Omega\EEE)\sm N$.
Then, by the assumption ($g_5$), for any fixed large $M>0$ there exists $t_M$ such that
\begin{equation}\label{2305201643}
g(x, [v_k], \nu_{v_k}) \geq M
  \chi_{ \{ |[v_k]|>t_M\} } \geq  M\, |\nu_{v_k} \cdot \xi_0|
  \chi_{ \{ |[v_k]\cdot \xi_0|>t_M\} } 
\end{equation}
for any $\xi_0 \in D$.
Arguing as in the case $g_\infty(x_0, \nu_0)\in \R$, with \eqref{2303201142} replaced by \eqref{2305201643}, we obtain that $\frac{\mathrm{d} \mu}{\mathrm{d} \hn}(x_0) \geq M\,|\nu_0 \cdot \xi_0|$ for every $M>0$ and $\xi_0 \in D$.
Taking $\xi_0$ such that $|\nu_0 \cdot \xi_0|\BBB > \EEE\tfrac12$ we obtain a contradiction for $M>2 \, \frac{\mathrm{d} \mu}{\mathrm{d} \hn}(x_0)$.
%
%
%
%
%
 Then $\partial^*\mathcal{P}\BBB \cap\Omega\EEE\subset N$,
and \eqref{2803201153} is proven.
 
 Therefore the general proof is concluded.
\end{proof}

\begin{remark}\label{rem:symjointlyconvex}
In \cite{FriPerSol20}  a class of functions $g\colon (\Rn)^3\to [0,+\infty)$
 satisfying for any bounded open set $\Omega\subset \Rn$
\begin{equation*}
\int_{J_v} g(v^+,v^-,\nu_v)\dh\leq \liminf_{h\to \infty} \int_{J_{v_h}} g(v_h^+,v_h^-,\nu_{v_h})\dh \quad\text{if }v_h\to v \text{ weakly in }GSBD^p(\Omega)
\end{equation*}
has been provided.
This is the class of \emph{symmetric jointly convex} functions, which are characterized by  (see \cite[Definition~3.1 and Theorem~5.1]{FriPerSol20})
\begin{equation*}
g(i,j,\nu)=\sup_{h\in \N}  (f_h(i)-f_h(j)) \cdot \nu \quad\text{for all }(i,j,\nu) \in (\Rn)^3 \text{ with }i\neq j
\end{equation*}
where $f_h\colon \Rn \to \Rn$ is a uniformly continuous, bounded, and conservative vector field (that is, there exists a potential $F_h\in C^1(\Rn)$ for which $\nabla F_h=f_h$) for every $h\in\N$.
Any symmetric jointly convex function depending only on the difference $i-j$ provides a function satisfying ($g_4$). Examples of such functions are (see \cite[Section~4]{FriPerSol20})
\begin{equation*}
g_1(y,\nu)=\wt g(|y|) \quad\text{for }\wt g\colon [0,+\infty) \to [0,+\infty) \text{ increasing with }\frac{\wt g(t)}{t} \text{ nonincreasing in }(0,+\infty)\,,
\end{equation*} 
\begin{equation*}
g_2(y,\nu)= \sup_{\{\xi_k\}_{k=1}^n \text{ orthornormal basis }} \Bigg(  \sum_{k=1}^n \theta_k \big( y\cdot \xi_k \big)^2 | \nu \cdot \xi_k|^2 \Bigg)^{1/2}
\end{equation*}
for $\theta_k \in C(\R;[0,+\infty))$ even and subadditive, for $k=1,\dots,n$ (this class has been introduced and studied in a $BD$ setting in \cite{DMOrlToaACV}), and
\begin{equation*}
g_3(y,\nu)=\psi(\nu)
\end{equation*}
for $\psi$ a norm. 
\end{remark}
From Theorem~\ref{thm:minimization} we deduce existence for the following minimisation problems with Dirichlet conditions, in the  propositions below. In the following $\Omega\subset \Rn$ is a bounded, open, connected and Lipschitz \BBB set. \EEE Moreover, we assume that $u_0\in W^{1,p}(\Rn;\Rn)$ and that $\dod \subset \dom$ be relatively open with $\dod=\wt \Omega \cap \dom$ for a bounded, open, connected domain $\wt\Omega\supset \Omega$.

 We consider first the simpler cases corresponding to $g_\infty\equiv +\infty$ and $g$ independent of the jump amplitude, in \BBB Propositions~\ref{prop:minimizzazioneNormaInfinita} and \ref{prop:minimizzazioneNormaFinita}. \EEE
Then we consider the case with general $f$, $g$ as in Theorem~\ref{thm:minimization}, which formally includes the other two cases. We prefer to state three different results since the proofs of \BBB Propositions~\ref{prop:minimizzazioneNormaInfinita} and \ref{prop:minimizzazioneNormaFinita} are more direct. \EEE
 \begin{proposition}\label{prop:minimizzazioneNormaInfinita} 
 Assume $f$, $g$ as in Theorem~\ref{thm:minimization}, with $g_\infty\equiv +\infty$.
 Then the problem
 \begin{equation}\label{DirichletMinimumPb}
 \min_{u=u_0 \in \wt\Omega\sm \ol \Omega} \bigg\{\int_\Omega f(x,e(u))\dx + \int_{J_{u}} g( x, \EEE [u], \nu_u) \dh \bigg\}
 \end{equation}
 admits a solution in $GSBD^p(\wt\Omega)$. In particular, this holds for $g(y,\nu)=\wt g(|y|)$ with  $\wt g\colon [0,+\infty) \to [0,+\infty)$ increasing, unbounded, and such that $\frac{\wt g(t)}{t}$ is nonincreasing.
 \end{proposition}
 \begin{proof}
 Let us apply Theorem~\ref{thm:minimization} to a minimising sequence $(u_h)_h$ for \BBB \eqref{DirichletMinimumPb} \EEE in $GSBD^p(\wt\Omega)$. 
 Since $g_\infty\equiv +\infty$, \EEE we have that the partition $\mathcal{P}$ of $\wt\Omega$ is trivial. The fact that $u_h=u_0$ in $\wt\Omega \sm \ol \Omega$ \BBB implies that one can choose $a_h$ (which for every $h$ reduces to a unique infinitesimal rigid motion) as $a_h=0$: in fact, if for $a_h^i=a_h$ and $a_h^j=0$  \eqref{2202201909} holds (namely, if $a_h$ and 0 are not in the same equivalence class, as discussed in the proof of Theorem~\ref{thm:main}), then $u_h-a_h$ would diverge on $\wt\Omega\sm \ol \Omega$, in contrast to the pointwise convergence toward a finite valued function in \eqref{2202201910}. \EEE
\BBB Therefore \EEE $u_h \to u$ in $L^0(\wt\Omega;\Rn)$.
 In particular $u_h=u_0$ in $\wt\Omega \sm \ol \Omega$ and
 using again that $g_\infty\equiv +\infty$ \EEE 
  we get 
 the lower semicontinuity of the functional $E$ to minimise. 
 \end{proof}
 \begin{remark}\label{rem:0304201045}
 In the above assumptions, if $g(y,\nu)\geq \wt c |y|$ for some $\wt c >0$, the solutions to \eqref{DirichletMinimumPb} belong to $SBD^p(\wt \Omega)$. In fact, this follows from the fact that $[u] \in L^1(J_u;\Rn)$: under such condition, every $GSBD$ function is in $SBD$, as shown in \cite[Theorem~2.9]{CC19ACV} (take $\mathbb{A}v=\mathrm{E}v$ therein, cf.\ \cite[Remark~2.5]{CC19ACV}).
 For other surface densities $g$, such as $g(y,\nu)=c+\sqrt{|y|}$, one obtains existence for the Dirichlet problem in $GSBD^p$.
 \end{remark}
 
 \begin{proposition}\label{prop:minimizzazioneNormaFinita} 
 Assume $f$ as in Theorem~\ref{thm:minimization}
  and let 
  $g\colon \Omega\times \Sn \to [c,+\infty)$ be continuous in the first variable and such that $g(x, \cdot)$ is a norm for every $x\in \Omega$.
 Then the problem
 \begin{equation}\label{DirichletMinimumPb'}
 \min_{u=u_0 \in \wt\Omega\sm \ol \Omega} \bigg\{ \int_\Omega f(x,e(u))\dx + \int_{J_{u}}  g(x, \nu_u) \EEE \dh \bigg\}
 \end{equation}
 admits a solution in $GSBD^p(\wt\Omega)$. 
 \end{proposition}
 \begin{proof}
 Let us apply Theorem~\ref{thm:minimization} to a
 minimising sequence $(u_h)_h$. 
 By \eqref{eqs:0203200917} we obtain a function $u \in GSBD^p(\wt\Omega)$, with $u=u_0$ in $\wt\Omega \sm \ol \Omega$. In fact, \BBB arguing as in the proof of Proposition~\ref{prop:minimizzazioneNormaInfinita} with \eqref{2202201909} and \eqref{2202201910}, the choice $a_h^1= 0$ is possible in the set $\wt\Omega\sm \ol\Omega$, which than has to be contained in a single element $P_1$ of the Caccioppoli partition. \EEE
 
 Moreover, 
 recalling Theorem~\ref{thm:Caccioppoli} we have that
 \begin{equation*}
 \begin{split}
 \int_\Omega f(x,e(u))\dx + \int_{J_{u}}  g(x, \nu_u) \EEE  \dh &\leq  \int_\Omega f(x,e(u))\dx + \int_{(J_{u}\cap \mathcal{P}^{(1)}) \cup (\partial^* \mathcal{P}\BBB \cap\wt\Omega\EEE)}  g(x, \nu_u) \EEE \dh
 \\& \leq \liminf_{h\to +\infty} \int_\Omega f(x,e(u_h))\dx + \int_{J_{u_h}}  g(x, \nu_{u_h}) \EEE \dh
 \\& = \inf_{u=u_0 \in \wt\Omega\sm \ol \Omega} \int_\Omega f(x,e(u))\dx + \int_{J_{u}}  g(x, \nu_u) \EEE \dh\,.
 \end{split}
 \end{equation*}
 Therefore $u$ is a solution to the problem \eqref{DirichletMinimumPb'} and this concludes the proof. \BBB We notice that by the chain of inequalities above, expressing the domain of the surface integral in the second inequality as the disjoint union  $(J_{u}\cap \mathcal{P}^{(1)}) \cup (J_u \cap \partial^* \mathcal{P})\cup \big((\partial^* \mathcal{P}\BBB \cap\wt\Omega)\sm J_u\big)=J_u \cup\big((\partial^* \mathcal{P}\BBB \cap\wt\Omega)\sm J_u\big)$, it holds 
 \[
 0=\int_{ (\partial^* \mathcal{P}\BBB \cap\wt\Omega\EEE)\sm J_{u}}  g(x, \nu_u) \EEE \dh \geq c \hn\big((\partial^* \mathcal{P}\BBB \cap\wt\Omega\EEE)\sm J_{u}\big)
 \]
 and then
 $\partial^* \mathcal{P}\BBB \cap\wt \Omega\EEE \subset J_u$, up to a $\hn$-negligible set.
 \end{proof} 
We consider now the case of general $g$.

\begin{proposition}\label{prop:1405201326}
Assume $f$, $g$ as in Theorem~\ref{thm:minimization}. Then the problem
\begin{equation*}\label{1405201336}
\min_{\substack{ u=u_0  \in \wt\Omega\sm \ol \Omega \\ \mathcal{P}\subset \wt\Omega \text{ Caccioppoli partition, } \partial^*\mathcal{P}\BBB \cap\wt\Omega\EEE\subset J_u}}  \hspace{-2em} \bigg\{ \hspace{-0.3em}\int_\Omega f(x,e(u))\dx + \int_{J_{u}\sm \partial^*\mathcal{P}} g( x, \EEE [u], \nu_u) \dh + \int_{\partial^*\mathcal{P}\BBB \cap\wt\Omega\EEE} \hspace{-1em} g_\infty( x, \EEE \nu) \dh    \bigg\}
\end{equation*}
admits a solution in $GSBD^p(\wt\Omega)$.
\end{proposition}
\begin{proof}
Let us denote 
\begin{equation*}
F(u, \mathcal{P}):=
\begin{dcases}
\hspace{-0.2em} \int_\Omega f(x,e(u))\dx \hspace{-0.2em}+ \hspace{-0.3em}\int_{J_{u}\sm \partial^*\mathcal{P}} g( x, \EEE [u], \nu_u) \dh \hspace{-0.2em}+ \hspace{-0.3em} \int_{\partial^*\mathcal{P}\BBB \cap\wt\Omega\EEE} \hspace{-1em} g_\infty( x, \EEE \nu) \dh &  \text{if }\partial^*\mathcal{P}\BBB \cap \wt\Omega\EEE\subset J_u,   \\
 +\infty &  \text{otherwise,}
 \end{dcases}
\end{equation*}
and fix a minimising sequence $(u_h, \mathcal{P}_h)$ for $F$. \BBB To shorten the notation, in the following of the proof we write simply $\partial^*\mathcal{P}$ in place of $\partial^*\mathcal{P}\cap \wt\Omega$. \EEE
We observe that, for $\mathcal{P}_h=(P_{h,j})_j$ we can assume that $\wt\Omega \sm \ol \Omega \subset P_{h,1}$ \BBB (arguing as in the proofs of Propositions~\ref{prop:minimizzazioneNormaInfinita}, \ref{prop:minimizzazioneNormaFinita}, and we may choose 0 as the infinitesimal rigid motion in $P_{h,1}$) \EEE and find piecewise rigid functions $\wt a_h$ such that 
\begin{equation}\label{1405201712}
\wt a_h=\sum_{j\in \N} \wt a_h^j \chi_{P_{h,j}}\,,
\end{equation}
\begin{equation}\label{1405201714}
\wt a_h^1= 0,\quad|\wt a_h^j(x)-\wt a_h^i(x)| \to +\infty \quad \text{for }\Ln\text{-a.e.\ }x\in \Omega, \text{ for all }i\neq j\,,
\end{equation}
and 
\begin{equation}\label{1405201344}
 E(u_h-\wt a_h) < F(u_h, \mathcal{P}_h) + \frac{1}{h}  
\end{equation}
for every $h \in \N$. (Recall \eqref{1705201919} for the definition of $E$).
In fact, since $u_h \in GSBD^p(\wt\Omega)$ it holds that $[u_h]\colon J_{u_h} \to \Rn$ is measurable, and then there exists for every $h$ a vanishing sequence $(s^h_k)_k$ such that 
\begin{equation}\label{1505201658}
\hn(J_{u_h}\BBB \cap \EEE \{ |[u_h]|>k \}) < s^h_k\,.
\end{equation}
%
Moreover, since $\hn(\partial^*\mathcal{P}_h)<+\infty$, for every $h,\,k\in \N$ there exists $m^h_k \in \N$ 
such that 
\begin{equation}\label{1505201659}
\sum_{j>m^h_k} \hn(\partial^* P_{h,j})<k^{-1}\,.
\end{equation}
Then we choose, in correspondence of $(P_{h,j})_j$, a sequence
$(\wt a_h^j)_j\subset \R^n$ (that is, each $\wt a_h^j$ is a constant function)  with $\wt a_h^1 = 0$ (in view of the Dirichlet boundary conditions), such that $|\wt a_h^j - \wt a_h^i|> 2k$ for $i \neq j \leq m^h_k$. By \eqref{1505201658}, \eqref{1505201659}, and triangle inequality  we find that 
\[
\hn(\partial^* \mathcal{P}_h\BBB \cap \EEE \{|[u_h - \wt a_h]|<k\})< s^h_k + k^{-1}\,.
\]
This implies, in view of ($g_5$) and since $g$ is a measurable function taking finite values, that there is $\ol k\in \N$, depending on $h$, large enough so that \eqref{1405201344} holds true.

Let us now apply Theorem~\ref{thm:minimization} to the sequence $(u_h-\wt a_h)_h\subset GSBD^p(\wt\Omega)$ (that satisfies the assumptions of Theorem~\ref{thm:minimization} by \eqref{1405201344}): this provides a function $u \in GSBD^p(\wt\Omega)$ and a sequence $(\widehat{a}_h)_h$ of piecewise rigid functions corresponding to a partition $\widehat{\mathcal{P}}=(\widehat{P}_j)_j$ (in particular, $J_{\widehat{a}_h}=\partial^*\widehat{\mathcal{P}}\BBB \cap \wt\Omega$\EEE) such that 
\[
u_h -\wt a_h -\widehat{a}_h \to u \quad\Ln\text{-a.e.}
\]
and
\begin{equation}\label{1405201903}
\int_\Omega f(x,e(u))\dx + \int_{J_u \cap \widehat{\mathcal{P}}^{(1)}} g( x, \EEE  [u], \nu_u) \dh + \int_{\partial^* \widehat{\mathcal{P}}\BBB \cap \wt\Omega\EEE} g_\infty( x,  \BBB \nu_{\widehat{\mathcal{P}}}) \EEE\dh\leq \liminf_{h\to +\infty} E(u_h-\wt a_h)\,.
\end{equation}
In particular, in view of the boundary conditions we may take $\wt\Omega\sm \ol \Omega\subset\widehat{P}_1$, $\widehat{a}_h^1= 0$ and we have that  $\partial^*\widehat{\mathcal{P}} \cap \wt \Omega \subset \ol \Omega$ and $u=u_0$ in $\wt\Omega\sm \ol \Omega$.
Collecting  \eqref{1405201344}, \eqref{1405201903}, and since $(u_h, \mathcal{P}_h)_h$ is a minimising sequence for $F$, we have that 
\begin{equation}\label{1405201922}
\int_\Omega f(x,e(u))\dx + \int_{J_u \cap \widehat{\mathcal{P}}^{(1)}} g( x, \EEE [u], \nu_u) \dh + \int_{\partial^* \widehat{\mathcal{P}}\BBB \cap \wt\Omega\EEE} g_\infty( x,  \BBB \nu_{\widehat{\mathcal{P}}}) \EEE \dh \leq\inf_{v,\mathcal{P}} F(v, \mathcal{P})\,.
\end{equation}
We notice now that we can find a piecewise rigid function $\wt a$ with $\wt a=0$ in $\wt\Omega\sm \ol \Omega$ and $J_{\wt a}\subset \partial^* \widehat{\mathcal{P}}\BBB \cap \wt\Omega\EEE$ for which $\partial^* \widehat{\mathcal{P}}\BBB \cap \wt\Omega\EEE\subset J_{u-\wt a}$. This follow from the fact that there are at most countable many $s \in \Rn$ such that $\hn(\partial^* \widehat{\mathcal{P}} \cap \{[u]=s\})>0$.
Moreover, since $\partial^* \widehat{\mathcal{P}}\BBB \cap \wt\Omega\EEE\subset J_{u-\wt a}$ (and in view of the fact that $g$ depends only on the jump amplitude, cf.\ below \eqref{1103200009}), we have that for $\widehat{u}=u-\wt a$
\begin{equation}\label{1405201923}
\int_\Omega f(x,e(u))\dx + \int_{J_u \cap \widehat{\mathcal{P}}^{(1)}} g( x, \EEE [u], \nu_u) \dh = \int_\Omega f(x,e(\widehat u))\dx + \int_{J_{\widehat u} \cap \widehat{\mathcal{P}}^{(1)}} g( x, \EEE [\widehat u ], \nu_{\widehat u}) \dh\,.
\end{equation}
Therefore, in view of the fact that  $\partial^* \widehat{\mathcal{P}}\BBB \cap \wt\Omega\EEE\subset J_{\widehat{u}}$,  by \eqref{1405201922} and \eqref{1405201923} we get that $(\widehat{u}, \widehat{\mathcal{P}})$ is a minimiser for $F$. This concludes the proof.
\end{proof}
\begin{remark}\label{1505202003}
The minimisation problem in Proposition~\ref{prop:1405201326} formally reduces to those one in Proposition~\ref{prop:minimizzazioneNormaInfinita} and \ref{prop:minimizzazioneNormaFinita} noticing that $\mathcal{P}=\{\wt\Omega\}$ when $g_\infty\equiv +\infty$ (cf.\ Theorem~\ref{thm:minimization})
and the functional in Proposition~\ref{prop:1405201326} does not depend on $\mathcal{P}$ when $g$ depends only on $\nu$ and coincides with $g_\infty$.
\end{remark}
\BBB \begin{remark}\label{rem:2608211805}
In the minimisation problem in Proposition~\ref{prop:1405201326} the restriction $\partial^*\mathcal{P}\BBB \cap\wt\Omega\EEE\subset J_u$ may be dropped. The mechanical interpretation for including this condition is to regard $\partial^*\mathcal{P}$ as part of the discontinuity set of the displacement, where the fracture is present, since also in $\partial^*\mathcal{P}$ the material can be interpreted as fractured.
\end{remark}
 \EEE

\bigskip
\noindent {\bf Acknowledgements.} The authors are grateful to the anonymous referees for their valuable and helpful comments, that permitted to greatly improve the paper.
For this project VC has received funding from the European Union’s Horizon 2020 research and innovation programme under the Marie Skłodowska-Curie grant agreement No. 793018. Most of this work was done while the authors were
still working in CMAP, CNRS and Ecole Polytechnique, Institut Polytechnique de
Paris, Palaiseau, France.

%

\Addresses

\end{document}